\documentclass[11pt,american]{article}
\PassOptionsToPackage{obeyFinal}{todonotes}
\usepackage{amsmath}
\usepackage{amsthm}
\usepackage{mathptmx}
\usepackage[scaled=0.92]{helvet}
\usepackage{courier}
\usepackage{newtxmath}
\usepackage[T1]{fontenc}
\usepackage[utf8]{inputenc}
\usepackage{geometry}
\usepackage{color}
\usepackage{babel}
\usepackage{verbatim}
\usepackage{refstyle}
\usepackage{textcomp}
\usepackage{mathtools}
\usepackage{todonotes}
\usepackage{microtype}
\usepackage{bbm}
\usepackage[unicode=true,pdfusetitle,
 bookmarks=true,bookmarksnumbered=false,bookmarksopen=false,
 breaklinks=true,pdfborder={0 0 0},pdfborderstyle={},backref=false,colorlinks=true]
 {hyperref}
\hypersetup{
 urlcolor=blue,linkcolor=blue,citecolor=blue}
 \usepackage[capitalize,nosort]{cleveref}
\crefdefaultlabelformat{#2\textup{#1}#3}
\crefformat{equation}{\textup{(#2#1#3)}}

\numberwithin{equation}{section}
\numberwithin{figure}{section}
\numberwithin{table}{section}
\theoremstyle{plain}
\newtheorem{thm}{\protect\theoremname}[section]
\theoremstyle{plain}
\newtheorem{prop}[thm]{\protect\propositionname}
\theoremstyle{remark}
\newtheorem{rem}[thm]{\protect\remarkname}
\theoremstyle{plain}
\newtheorem{lem}[thm]{\protect\lemmaname}
\theoremstyle{plain}
\newtheorem{cor}[thm]{\protect\corollaryname}

\hypersetup{final}

\makeatother

\providecommand{\corollaryname}{Corollary}
\providecommand{\lemmaname}{Lemma}
\providecommand{\propositionname}{Proposition}
\providecommand{\remarkname}{Remark}
\providecommand{\theoremname}{Theorem}

\newcommand{\email}[1]{{\href{mailto:#1}{\nolinkurl{#1}}}}
\begin{document}
\global\long\def\Cov{\operatorname{Cov}}%

\global\long\def\Var{\operatorname{Var}}%

\global\long\def\Lip{\operatorname{Lip}}%

\global\long\def\e{\mathrm{e}}%

\global\long\def\ii{\mathrm{i}}%

\global\long\def\R{\mathbf{R}}%

\global\long\def\Law{\operatorname{Law}}%

\global\long\def\sgn{\operatorname{sgn}}%

\global\long\def\supp{\operatorname{supp}}%

\global\long\def\diam{\operatorname{diam}}%

\global\long\def\spn{\operatorname{span}}%

\global\long\def\Hess{\operatorname{Hess}}%

\global\long\def\adj{\operatorname{adj}}%

\global\long\def\tr{\operatorname{tr}}%

\global\long\def\esssup{\operatorname*{ess\,sup}}%

\global\long\def\dif{\mathrm{d}}%

\global\long\def\eps{\varepsilon}%
\global\long\def\1{\mathbbm{1}}%
\global\long\def\A{\mathcal{A}}%

\title{A quenched local limit theorem for stochastic flows}
\author{Alexander Dunlap\thanks{Department of Mathematics, Courant Institute of Mathematical Sciences, New York University, New York, NY 10012 USA.             \email{alexander.dunlap@cims.nyu.edu}.} \and Yu
  Gu\thanks{Department of Mathematics, Carnegie Mellon University, Pittsburgh,  PA 15213 USA. \emph{Current address:} Department of Mathematics, University of Maryland, College Park, MD 20742 USA. \email{ygu7@umd.edu}.}}
\maketitle
\begin{abstract}
	We consider a particle undergoing Brownian motion in Euclidean space
	of any dimension, forced by a Gaussian random velocity field that
	is white in time and smooth in space. %
	We show that conditional
	on the velocity field, the quenched density of the particle after a long time can be approximated pointwise
	by
	the product of a deterministic Gaussian density and a spacetime-stationary
  random field $U$. If the velocity field is additionally assumed to be incompressible, then $U\equiv 1$ almost surely and we obtain a local central limit theorem.
\end{abstract}

\section{Introduction}

Let $(\Omega,\mathcal{H},\mathbf{P})$ be a probability space. %
Fix
a spatial dimension $d\in\mathbf{N}$. Let $V=(V_{1},\ldots,V_{d})$
be a   Wiener process on $L^{2}(\mathbf{R}^{d};\mathbf{R}^{d})$ that is spatially-smooth,
with covariance function formally given by
\[
	\mathbf{E}V_{i}(\dif t,x) V_{j}(\dif s,y)=\delta(t-s)R_{ij}(x-y)\dif s\dif t
\]
for some covariance function $R\in\mathcal{C}_{\mathrm{c}}^{\infty}(\mathbf{R}^{d};\mathbf{R}^{d\times d})$.
(That is, $R$ is smooth with compact support.) Let $\{\mathcal{F}_{t}\}$
be the usual filtration associated to $V$ (generated by $\{V(s)\ :\ s\le t\}$)
and let $\mathcal{F}=\bigvee_{t<\infty}\mathcal{F}_{t}$. Let $B=(B_{1},\ldots,B_{d})$
be a Brownian motion taking values in $\mathbf{R}^{d}$ (independent
of $\mathcal{F}$) with quadratic variation
\begin{equation}
	\langle B_{i},B_{j}\rangle(t)=\nu\delta_{ij}t\label{eq:Btcovkernel}
\end{equation}
for some $\nu>0$.
Let $\{\mathcal{G}_{t}\}$ be the usual filtration associated to
$B$, and let $\mathcal{G}=\bigvee_{t<\infty}\mathcal{G}_{t}$. Let
$\mathcal{H}_{t}=\mathcal{F}_{t}\vee\mathcal{G}_{t}$. We assume that the $\sigma$-algebra $\mathcal{H}$ is given by $\mathcal{H}=\mathcal{F}\vee\mathcal{G}$.

We are interested in the stochastic differential equation
\begin{align}
	\dif X(t) & = V(\dif t,X(t))+ \dif B(t);\label{eq:dX} \\
	X(0)      & =0,\label{eq:Xic}
\end{align}
which models a passive scalar in a environment that decorrelates rapidly in time.
We will interpret \cref{eq:dX} in the manner of \cite[Section 3.4]{Kun97};
that is, as equivalent to the \emph{Itô} integral equation
\[
	X(t)=B(t)+\int_{0}^{t} V(\dif s,X(s)),
\]
where $\{X(t)\}$ is assumed to be a continuous $\mathbf{R}^{d}$-valued
process adapted to $\{\mathcal{H}_{t}\}$.

This problem has a unique solution by \cite[Theorem 3.4.1]{Kun97}
(using our assumption on the smoothness of $R$). The process $\{X(t)\}$ is a continuous martingale with
quadratic variation process given by
\begin{equation}
	\langle X_{i},X_{j}\rangle(t)=(\nu\delta_{ij}+R_{ij}(0))t\label{eq:QVprocess},
\end{equation}
by \cite[Theorem 3.2.4]{Kun97}. Thus, the annealed law of $\{X(t)\}$ is actually a $d$-dimensional Brownian motion with covariance matrix $(\nu I_{d}+R(0))t$ at time $t$. Here we used $I_d$ to denote the $d\times d$ identity matrix. We will think of the forcing $V$
as a random velocity field and the forcing $B$ as a molecular diffusion,
so $\nu$ is the ``molecular diffusivity''. Our interest will be
in the quenched (with respect to the molecular diffusion)
law
of $X$ given by
\begin{equation}
	\mu_{t}=\Law[X(t)\mid\mathcal{F}].\label{eq:mutfromdelta}
\end{equation}
We will show in \cref{subsec:stochPDE} that, for $t>0$,  $\mu_{t}$ has a density
with respect to Lebesgue measure on\textbf{ $\mathbf{R}^{d}$} that
exists as a random field $\big(u(t,x)\big)_{t>0,x\in\mathbf{R}^{d}}$ (as a consequence of 
the molecular diffusion):
\begin{equation}\label{e.defdensity}
	\mu_t(\dif x)=u(t,x)\dif x,  \quad\quad t>0,  x\in\R^d.
\end{equation}
Thus, $u(t,\cdot)$ is a density function that feels the randomness of the velocity field. Let
$G_t$ be the solution to the PDE
\begin{equation}\label{eq:eqG}
	\begin{aligned}
		\partial_{t}G_t(x) & =\frac{1}{2}\tr[(\nu I_{d}+R(0))\nabla^{2}G_t(x)]; \\
		G_0                & =\delta_{0}.
	\end{aligned}
\end{equation}
Here and
throughout the paper, we use $\nabla^{2}$ to mean the Hessian operator,
\emph{not} the Laplacian.
Thus
\[G_t(x) = \frac1{[(2\pi t)^{d}\det (\nu I_d+R(0))]^{1/2}}\exp\left\{-\frac1{2t}x\cdot (\nu I_d+R(0))^{-1}x\right\}
\]
is the Gaussian density centered at the origin with
covariance matrix $(\nu I_{d}+R(0))t$, and so $G_t$ is the density of annealed law of $X(t)$. The goal of this paper is
to study the relationship between the quenched law $u(t,\cdot)$ and the annealed law $G_t(\cdot)$, and to understand how the randomness from the environment affects the \emph{local} behavior of the passive scalar. Here is the main theorem:
\begin{thm}
	\label{thm:maintheorem}There is a spacetime-stationary random field
	$U$, positive almost surely with $\mathbf{E}U\equiv1$, and, for every $\eps>0$,
	a constant $C=C(R,\nu,\eps)<\infty$ so that
		\begin{equation}
			\sup_{x\in\mathbf{R}^{d}}\mathbf{E}|u(t,x)-G_t(x)U(t,x)|^{2}\le  Ct^{-d}\bigg(t^{-1/3}\log t\1_{d=1}+t^{-2/(2+d)+\eps}\1_{d\ge 2}\bigg) %
			\label{eq:mainthm-bound}
		\end{equation}
		for all $t\ge C$. In particular, for any $c<\frac13\1_{d=1}+\frac{2}{2+d}\1_{d\ge 2}$, we have
		\begin{equation}
			\lim_{t\to\infty}\sup\left\{\mathbf{E}\left|\frac{u(t,x)}{G_t(x)}-U(t,x)\right|^{2}\ :\ (\nu I_{d}+R(0))^{-1}x\cdot x\le ct\log t\right\}=0.\label{eq:takethesup}
		\end{equation}
		We in fact have $U\equiv 1$ almost surely if and only if $\sum_{i=1}^d \frac{\partial R_{ij}}{\partial x_i}\equiv 0$ for each $j$ (which holds if and only if $V$ is incompressible almost surely). In general, we have, for all $t\in\R$ and $x_1,x_2\in\R^d$, that
		\begin{equation}
		 |\mathbf{E}[(U(t,x_1)-1)(U(t,x_2)-1)]|\le C(1+|x_1-x_2|)^{-d}\label{eq:Ucorrbd}
		\end{equation}
		and, for all $t_1,t_2\in\R$ and $x_1,x_2\in\R^d$, that
		\begin{equation}
		 |\mathbf{E}[(U(t_1,x_1)-1)(U(t_2,x_2)-1)]|\le C(1+|t_1-t_2|)^{-(d-\eps)/4}.\label{eq:Utimecorrbd}
		\end{equation}
\end{thm}

For $\delta>0$, define $X_{\delta}(t)=\delta X(\tfrac{t}{\delta^2})$, so the quenched density of $X_{\delta}(t)$ is
\[
	u_{\delta}(t,x)={\delta}^{-d}u(\tfrac{t}{{\delta}^2},\tfrac{x}{\delta}).
\]
By \cref{eq:mainthm-bound}, we have for any $t>0,x\in\R^d$, that
\[
	\mathbf{E}|u_{\delta}(t,x)- G_t(x) U(\tfrac{t}{{\delta}^2},\tfrac{x}{\delta})|^2\to0, \text{ as }\delta\to0.
\]
In other words, the quenched density of the diffusively rescaled process is approximately the Gaussian density multiplied by a stationary random field. The stationary random field can thus be viewed as the ``corrector'' in stochastic homogenization, in the sense that it multiplies a homogenized field to give the exact field. This corrector is the  constant $1$ in the incompressible case, so we obtain a local central limit theorem.

Results similar to \cref{thm:maintheorem} were proved for discrete-time random walk in a discrete-space random environment in \cite{BMP99, DeGu19}.
(See also the survey \cite[§1.4.3]{BMP07} regarding the result of \cite{BMP99}.)
Since the result concerns the long-time behavior of the system, one
does not expect a substantial difference between the discrete and
continuous settings. However, the local temporal roughness of the
driving force introduces substantial complications in establishing
the required estimates, as we discuss in \cref{subsec:proofstrategy}
below. Moreover, \cref{thm:maintheorem} provides a quantitative rate of convergence, and it is meaningful in the entire diffusive bulk region (i.e. $|x|\lesssim t^{1/2}$), while \cite[Theorem 2]{BMP99} only  holds for $|x|\ll t^{1/3}$. Similar results were also shown in \cite{TLD16,BW21} for certain exactly-solvable models, with the one-point distribution of the correction field $U$ characterized explicitly. It is also worth mentioning that for reversible random walks/diffusions in random environments, e.g., the random conductance model, one can actually prove the local central limit theorem. Using our notation this says that $u_{\delta}(t,x)\approx G_t(x)$ for $\delta\ll1$, similar to our result when $V$ is incompressible. We refer the reader to \cite{andres2014invariance,andres2021quenched} and the references therein.

The stationary random field $U$ is a spacetime stationary solution to the Fokker-Planck \eqref{utdef} below with $\mathbf{E} U=1$, which is closely related to the invariant measure of the process of ``environment seen from the particle,'' a crucial object in the study of random walk/diffusion in random environment. This connection was also made in \cite{DeGu19} for random walks in a balanced random environment. We refer to \cref{rem:environment} for more discussion.

Our interest in the quenched density $u(t,x)$ is motivated in part by the recent work on the moderate- and
large-deviations regime of diffusion in a time-dependent random environment, which decorrelates rapidly; see the discussions \cite{BC17,BLD20,barraquand2020large,LDT17} in both the physics and
mathematics literature. In the diffusive regime $x\sim \sqrt{t}$, which is what we consider here, it is well-known that the diffusion scales to a Brownian motion, see e.g. the discussion on similar models in \cite{boldrighini2004random,rassoul2005almost, fannjiang1999turbulent,komorowski2001homogenization} and \cite{komorowski2012fluctuations} for a
monograph on the subject.  (In our special setting of white-in-time noise, the annealed law of $X(t)$ is actually exactly the Brownian motion, and using our main result, it is not difficult to show that the quenched law is approximately the Brownian motion on the diffusive scale.) To study the error, one can consider quantities of the form
\[
	\mathbf{E}[f(\eps X(\tfrac{t}{\eps^2}))\mid\mathcal{F}]-\int_{\R^d}f(x)G_t(x)\dif x=\int_{\R^d} f(x) [u_\eps(t,x)-G_t(x)]\dif x,
\]
where $f:\R^d\to\R$ is an arbitrary smooth function. The Edwards-Wilkinson type fluctuation is proved in \cite{BRAS06,BP01,yu2016edwards}, i.e., after a proper rescaling, $\{u_\eps(t,x)-G_t(x)\}_{t>0,x\in\R^d}$ converges in law and weakly in space to a Gaussian field that solves a stochastic heat equation with an additive Gaussian noise. Compared to our result, the difference is that we consider the  fluctuation   $u_\eps(t,x)-G_t(x)$  for any fixed $(t,x)$, rather than performing a spatial averaging under which the local fluctuations average out so that one needs to consider the next order error to observe random fluctuations. One can also look at super-diffusive regimes. In the moderate-deviations
regime of $x\sim t^{3/4}$, the KPZ equation arises \cite{BLD20} (see \cite{corwin2017kardar} for a similar result in a weak noise but large-deviations regime). The large-deviations regime $x\sim t$
is associated with the KPZ fixed point, and the Tracy-Widom type distribution was derived in  \cite{BC17}.

In \cite{LDT17}, the relation between the diffusion in time-dependent random environments and the KPZ universality class was explored. For $\log u(t,x)$, the Edwards--Wilkinson universality was actually conjectured to prevail in the diffusive regime, and it was also pointed out that the expected normal statistics seems to be different from the one studied in \cite{BRAS06}. Our result of $u(t,x)\approx G_t(x)U(t,x)$ in $t\gg1$ shows that the random fluctuation is governed by the stationary random field $U$, but we do not  observe log-normal fluctuations of $U(t,x)$. Instead, as it will become clear later in the proof, $U(t,x)$ is  a deterministic functional of the local random environment near $(t,x)$, so there is actually no averaging taking place.  It is very similar to the case of a directed polymer in a random environment in dimension $d\geq3$ at high temperature, where it is well-known that the polymer path is diffusive and the partition function is approximately a deterministic functional of the random environment near the endpoint.

We approach the problem from a more analytic perspective.
We will show in \cref{prop:utsolvesSPDE} below that $u$ satisfies
the stochastic PDE
\begin{align}
	\dif u(t) & =\frac{1}{2}\tr[(\nu I_{d}+R(0))\nabla^{2}u(t)]\dif t-\nabla\cdot[u(t) V(\dif t)],\qquad t>0;\label{eq:utdef-intro} \\
	u(0)      & =\delta_{0},\label{eq:uic-intro}
\end{align}
which can be seen as a Fokker--Planck equation with random coefficients. Then the field $U$ in \cref{thm:maintheorem}
is in fact a spacetime-stationary solution to \cref{eq:utdef-intro}, starting from constant initial data $u(0,x)\equiv1$.
Thus, \cref{thm:maintheorem} is quite similar to the ``homogenization-type''
theorems of \cite{DGRZ18,CCM20} proved for the stochastic heat equation
with weak noise in $d\ge3$, in that it shows how to approximate the
solution to a stochastic PDE with a compactly-supported initial condition
by a deterministic evolution multiplied by a random spacetime-stationary
solution. (A similar result was proved for directed polymers in $d\ge3$
in \cite{BMP97}.)

In the case when the forcing is assumed to be incompressible (i.e.
$\nabla\cdot V \equiv 0$ almost surely), the SPDE \cref{eq:utdef-intro} has been
extensively studied in the turbulence community as the ``rapid decorrelation
in time model'' or ``Kraichnan model.'' See \cite{MK99} and the
references therein. %

\subsection{Proof strategy\label{subsec:proofstrategy}}

As pointed out above, our result is quite similar in form to the
results on the stochastic heat equation in $d\ge3$. %
If we ignore convergence issues and formally write the mild solution
formula to \cref{eq:utdef-intro}
\begin{equation}
	\begin{aligned}u(t) & =G_{t}*u(0)-\int_{0}^{t}\int G_{t-s}*\nabla\cdot[u(s)V(\dif s)] \\
                    & =G_{t}*u(0)-\int_{0}^{t} \nabla G_{t-s}* [u(s)V(\dif s)],
	\end{aligned}
	\label{eq:mildsolution}
\end{equation}
then we immediately see the similarity between \cref{eq:utdef-intro}
and the stochastic heat equation, with the only difference coming from
the use of $\nabla G_{t-s}$ instead of $G_{t-s}$ in the stochastic integral term.
This extra gradient is the reason our result holds in $d\ge1$, rather
than the requirement of $d\ge3$ for the stochastic heat equation. To see it more clearly, one can look at the first order ``chaos'', which is the first random term obtained by iterating the mild formulation: for SHE, we obtain $ \int_0^t G_{t-s}* V(\dif s)$, which converges to a stationary Gaussian field in large time, only in $d\geq3$; for the Fokker-Planck equation, the convergence of $\int_0^t\nabla G_{t-s}* V(\dif s)$ to a stationary Gaussian field holds in any dimension.
The extra gradient also means that making \cref{eq:mildsolution} rigorous
is quite nontrivial, due to the worse singularity of $\nabla G_{t}(x)$
near $(t,x)=(0,0)$. (Here we mention a recent work \cite{KH20} for a special class of $V$.) Thus, we do
not use the formulation \cref{eq:mildsolution} in the present work,
and instead use another approach to make sense of the SPDE \cref{eq:utdef-intro}.

While it is not difficult to formally derive \cref{eq:utdef-intro}
as the Fokker--Planck equation associated with the passive scalar evolution \cref{eq:dX},
solution theories for the stochastic PDE \cref{eq:utdef-intro} are
rather intricate; see the discussion in \cite[pp.~2–3]{KH20}. We
will use a solution theory due to Kunita \cite{Kun94gsaspde} (similar
to the approach described in \cite[§6.2]{Kun97}) that uses stochastic
flows to make sense of the stochastic PDE. We note that we require a somewhat stronger solution theory than simply deriving the problem \cref{eq:utdef-intro}--\cref{eq:uic-intro}
solved by the density, because, as indicated above, we will also need to
construct spacetime-stationary solutions to \cref{eq:utdef-intro}, with the initial data $u(0,x)\equiv1$. We recall the results we will need in \cref{subsec:stochPDE}.
This approach requires $R$ to be (qualitatively)
several times differentiable, which we have assumed in our work. Alleviating
this restriction was part of the goal of \cite{KH20}, but results in this
direction are not yet strong enough for our purposes.

To justify the approximation
\begin{equation}\label{eq:uGU}
	u(t,x)\approx G_t(x) U(t,x),   \quad\quad t\gg1,
\end{equation} and thus prove \cref{thm:maintheorem},  our strategy is similar to that of
\cite{DG20} for the 2D nonlinear stochastic heat equation. Namely, we first
approximate \cref{eq:utdef-intro} by the equation for
which the noise has been turned off in the time interval $[0,q]$, for some properly chosen $q$ so that $1\ll t-q\ll t$.  Then we show that the latter
solution can be approximated locally in space by a stationary
solution. Basically, the evolution of \cref{eq:utdef-intro} in the time interval $[0,q]$, which is almost of length $t$, generates the factor $G_t(x)$ in \cref{eq:uGU}, while the evolution in the remaining interval $[q,t]$, which is  macroscopically small but microscopically large, ``feels'' the random environment and produces the factor $U(t,x)$ in \cref{eq:uGU}. A difference is that \cite{DG20}
works with a stochastic heat equation in $d=2$, where spacetime-stationary
solutions do not exist. Thus, as we have stated before, phenomenologically
the situation is more similar to that considered in \cite{DGRZ18,CCM20},
although in those works a different approach based on the Feynman--Kac
formula was used in the proofs.

Proving the mentioned bounds in \cite{DG20} was done using the mild
solution formula, the analogue of \cref{eq:mildsolution}. A discrete
chaos expansion was also the key technique used for the proof in \cite{BMP99}.
As we have stated, we do not (at present) have a mild solution theory
for the SPDE \cref{eq:utdef-intro}. Thus we work in a more analytic way, using the PDE satisfied by  the two-point
correlation function of the solution to \cref{eq:utdef-intro} in \cref{subsec:second-moment-PDE}.
This PDE has been used before in the case of the Kraichnan model (i.e.
when the forcing is assumed incompressible); see for example \cite{Maj93}. Then we use tools from the theory of parabolic PDE (in particular \cite{Fri64,Esc00}) to prove the required bounds on the correlations. We establish these bounds in \cref{subsec:PDEbounds}. Then we apply them in \cref{sec:stationary}
to prove the existence of the spacetime-stationary solution $U$ and in
\cref{sec:deltaic} to complete the proof of \cref{thm:maintheorem}.

\subsection{Acknowledgments}

A.D. was partially supported by the NSF Mathematical Sciences Postdoctoral Fellowship program via grant no.\ DMS-2002118. Y.G. was partially supported by the NSF through DMS-1907928 and DMS-2042384 and the Center for Nonlinear Analysis of CMU. We thank Guillaume Barraquand for helpful comments on a draft of the manuscript, Xiaoqin Guo for discussions, and Margaret Smith at NYU Libraries for help in obtaining copies
of references during the COVID-19 pandemic.

\section{Setup and preliminaries\label{sec:setup}}

Throughout the paper, the letter $C$ will denote a positive constant
depending on $R$ and $\nu$, and only on other parameters if specified
explicitly. We will allow $C$ to change from line to line
if necessary.

We wish to derive a stochastic PDE satisfied by $\mu_{t}$, but before
we do this we will generalize \cref{eq:dX}--\cref{eq:Xic} to the setting
of \emph{stochastic flows} (see \cite[Chapter 4]{Kun97}). Let $\varphi_{s,t}(x)$
($s,t\in\mathbf{R}$, $x\in\mathbf{R}^{d}$) be the family of random
diffeomorphisms solving the family of SDEs
\begin{align}
	\dif_t\varphi_{s, t}(x) & =  V(\dif t,\varphi_{s,t}(x))+ \dif B(t);\label{eq:flowSDE} \\
	\varphi_{s,s}(x)        & =x,\label{eq:flowIC}
\end{align}
by which we mean solving the stochastic Itô integral equations
\begin{equation}\label{eq:551}
	\varphi_{s,t}(x)=x+\int_{s}^{t}[V(\dif r,\varphi_{s,r}(x))+\dif B( r)], \quad\quad t\geq s.
\end{equation}
This means that the solution  to \cref{eq:dX}--\cref{eq:Xic} will
be given by $X(t)=\varphi_{0,t}(0)$. Such a solution exists and is unique
by \cite[Theorem 4.5.1]{Kun97}.

\subsection{The stochastic PDE\label{subsec:stochPDE}}

Now for a Borel measure $\mu_{0}$ on $\mathbf{R}^{d}$, which we
assume to live in some weighted Sobolev space (of negative regularity)
with at most polynomial growth at infinity, let $\tilde{\mu}_{t}$   %
be the pushforward measure of $\mu_{0}$ by $\varphi_{0,t}$, so for any
$A\subset\mathbf{R}^{d}$, we have
\begin{equation}
	\tilde{\mu}_{t}(A)=\mu_{0}(\varphi_{0,t}^{-1}(A)).\label{eq:utildedef}
\end{equation}
Thus, $\tilde{\mu}_{t}$ is an $\mathcal{H}_{t}$-measurable random
measure.

The definition \cref{eq:utildedef} is similar to \cite[(2.14)]{Kun94sfaosd}
and \cite[(2.4)]{Kun94gsaspde}, which define the composition of a
tempered distribution and a stochastic flow. We emphasize, however,
that the composition of a tempered distribution with a diffeomorphism
is \emph{not} a generalization of the pushforward of a measure by
a diffeomorphism, as the former construction involves a factor of
the Jacobian determinant of the diffeomorphism. That is, our definition
\cref{eq:utildedef} is in fact the same as defining
\begin{equation}
	\tilde{\mu}_{t}=\left(\frac{\mu_{0}}{\det D\varphi_{0,t}}\right)\circ\varphi_{0,t}^{-1},\label{eq:utildedef-tempereddistn}
\end{equation}
where the $\circ$ denotes composition of distributions, in the sense
that
\[
	\langle\tilde{\mu}_{t},f\rangle=%
	\int\frac{1}{\det D\varphi_{0,t}(x)}(\det D\varphi_{0,t}(x))f(\varphi_{0,t}(x))\,\dif\mu_{0}(x)=\int f(\varphi_{0,t}(x))\,\dif\mu_{0}(x),
\]
which agrees with \cref{eq:utildedef}. The determinants involved in
the last two formulas are positive, so there is no need to take an
absolute value.

Now we define
\begin{equation}
	\mu_{t}=\mathbf{E}[\tilde{\mu}_{t}\mid\mathcal{F}],\label{eq:takeconditionalexpectation}
\end{equation}
so \cref{eq:mutfromdelta} represents the special case when $\mu_{0}=\delta_{0}$.
Conditional expectations of the form \cref{eq:takeconditionalexpectation}
were constructed and studied in \cite{Kun94gsaspde}. By \cite[Theorem 3.2]{Kun94gsaspde}
(which relies on the partial Malliavin calculus developed in \cite{BM81,KS84}),
for all $t>0$ the measure $\mu_{t}$ has a (spatially) smooth density
$u(t)$ with respect to the Lebesgue measure almost surely. This
property comes from the ellipticity implied by \cref{eq:Btcovkernel}
of the molecular diffusion. The following proposition shows that \cref{eq:takeconditionalexpectation} solves the Fokker-Planck in an appropriate sense:
\begin{prop}
	\label{prop:utsolvesSPDE}The function $u$, considered as a time-indexed family of tempered
	distributions on\textbf{ $\mathbf{R}^{d}$}, is the unique solution
	of the Itô stochastic PDE
	\begin{align}
		\dif u( t)              & =\frac{1}{2}\tr[(\nu I_{d}+R(0))\nabla^{2}u(t)]\dif t-\nabla\cdot[u(t)V(\dif t)],\qquad t>0;\label{eq:utdef} \\
		\lim_{t\downarrow0}u(t) & =\mu(0)\label{eq:ic}
	\end{align}
	in the ``generalized solution'' sense analogous to \cite[(3.3)]{Kun94gsaspde}: for almost every realization of the random environment, we have for all Schwartz functions $h:\R^d\to\R$ that
		\begin{equation}
			\langle u(t),h\rangle =\langle\mu(0),h\rangle+\int_0^t\frac12\langle u(s),\tr[(\nu I_{d}+R(0))\nabla^{2}h]\rangle\,\dif s+\int_0^t\langle u(s),\nabla h\cdot V(\dif s)\rangle, \quad\quad t>0.\label{eq:generalizedsolution}
		\end{equation}
\end{prop}

\begin{rem}
	In the sequel, we will use the standard abuse of notation and write
	\cref{eq:ic} as $u(0)=\mu(0)$, even if $\mu(0)$ does not have
	a density.
\end{rem}

\begin{proof}
	\newcommand{\kunitanotation}{\mathsf}
	We will derive \cref{eq:utdef} by applying \cite[Theorem~3.1]{Kun94gsaspde}. In order to use this theorem, we must show how our problem fits into the framework of \cite{Kun94gsaspde}. This is done via the following list of correspondences, in which the left-side quantities (also written in sans-serif type to avoid confusion with the notation used in the present paper) are the notations of \cite{Kun94gsaspde} and the right-side quantities are our notations:
	\begin{align}
		\kunitanotation{W}^i     & = -V_i,\qquad 1\le i\le d;\label{eq:Wi}                                                                                                                                                                                                                 \\
		\kunitanotation{W}^{d+1} & = -\nabla\cdot V;\label{eq:Wd1}                                                                                                                                                                                                                         \\
		\kunitanotation{B}(x,t)  & =-B(t)+\frac12(\nabla\cdot R)(0)t;\label{eq:B}                                                                                                                                                                                                          \\
		\kunitanotation{L}       & = \frac12\nu \Delta + \frac12(\nabla\cdot R)(0)\cdot\nabla, \text{ i.e. } \kunitanotation{a}^{ij}\equiv\nu\delta_{ij},\ \kunitanotation{b}^i \equiv \frac12\sum_{j=1}^d\frac{\partial R_{ij}}{\partial x_j}(0),\ \kunitanotation{d}\equiv0;\label{eq:L} \\
		\kunitanotation{X}       & = \mu(0);\label{eq:X}                                                                                                                                                                                                                                   \\
		\kunitanotation{G}       & \equiv 0.\label{eq:Gd}
	\end{align}
	Here and henceforth, by $(\nabla\cdot R)(0)$ we denote the vector with the $i$th component $\sum_{j=1}^{d}\frac{\partial R_{ij}}{\partial x_{j}}(0)$.

	As \cite{Kun94gsaspde} works with Stratonovich rather than Itô integrals,
	we rewrite \cref{eq:flowSDE} in the Stratonovich form. %
	Using the Itô--Stratonovich
	correction given in \cite[Theorem 3.2.5]{Kun97}, we have
	\begin{align}
		\varphi_{s,t}(x) & =x+\int_s^t V( \dif r, \varphi_{s,r}(x))+B(t)-B(s)\nonumber \\&=x+\int_s^t V(\circ \dif r,\varphi_{s,r}(x))+B(t)-B(s)-\frac12(\nabla\cdot R)(0)(t-s),\label{eq:phistratonovich}
	\end{align}
	or equivalently
	\[
		\dif_t\varphi_{s,t}(x)=V(\circ \dif t,\varphi_{s,t}(x))+B(\dif t)-\frac{1}{2}(\nabla\cdot R)(0)\dif t.
	\]
	Using the correspondences \cref{eq:Wi}--\cref{eq:L}, we see that \cref{eq:phistratonovich} matches \cite[(3.4)]{Kun94gsaspde}, with $\kunitanotation{\psi}=\varphi$.

	Now using
	the differentiation rule \cite[(3.3.21)]{Kun97} we have that
	\[
		\dif_t[D\varphi_{s,t}(x)]=\dif(DV)(\circ\dif t,\varphi_{s,t}(x))\cdot D\varphi_{s,t}(x).
	\]
	Therefore, by the Jacobi formula and the chain rule for Stratonovich
	integrals, we have
	\begin{align*}
		\dif_t[\det D\varphi_{s,t}(x)] & =\tr[(\adj D\varphi_{s,t}(x))DV(\circ\dif t,\varphi_{s,t}(x))\cdot D\varphi_{s,t}(x)] \\
		                               & =\det D\varphi_{s,t}(x)(\nabla\cdot V)(\circ\dif t,\varphi_{s,t}(x)),
	\end{align*}
	where $\adj$ denotes the classical adjoint (adjugate) matrix. This implies that
	\[
		\dif_t[\log\det D\varphi_{s,t}(x)]=(\nabla\cdot V)(\circ\dif t,\varphi_{s,t}(x)),
	\]
	so
	\begin{equation}
		\det D\varphi_{s,t}(x)=\exp\left\{ \int_{0}^{t}(\nabla\cdot V)(\circ\dif r,\varphi_{s,r}(x))\right\},\label{eq:detasexp}
	\end{equation}%
	Therefore, recalling \cref{eq:Wd1} and \cref{eq:L}, we have
	\begin{equation}
		\kunitanotation{\gamma}_{s,t}(x) = \frac{1}{\det D\varphi_{s,t}(x)},\qquad \kunitanotation{X}(t) = \frac{\mu(0)}{\det D\varphi_{0,t}},
	\end{equation}
	with the left sides in the notation of \cite[(3.5)--(3.6)]{Kun94gsaspde} and the right sides in our notation.

	Now we see that \cite[Theorem~3.1]{Kun94gsaspde} applies, and it tells us that %
	\begin{equation}
		\begin{aligned}
			\mu(t) & = \mu(0)+\int_0^t \left[\frac12\nu\Delta \mu(s)+\frac{1}{2}(\nabla\cdot R)(0)\cdot \nabla \mu(s)\right]\,\dif s -\int_0^t \nabla \mu(s)\cdot V(\circ \dif s) \\&\qquad-\int_0^t\mu(s)(\nabla\cdot V)(\circ\dif s).
		\end{aligned}\label{eq:mustratproblem}
	\end{equation}
	Since $u(t)$ is the density of $\mu(t)$,
	the same equation holds for $u$.

	To complete the proof, it remains to convert \cref{eq:mustratproblem} into an Itô integral equation by subtracting the appropriate correction term. This computation is carried out on in \cite[p.~302]{Kun97}. Again using a sans-serif font for the notation there, we have $\kunitanotation{F^i} = -V_i$ for $i=1,\ldots,d$ and $\kunitanotation{F^{d+1}} = -\nabla\cdot V$. Thus we have the ``local characteristic''
	\begin{equation}
		\kunitanotation{A^{ij}}(x,y,t) = \begin{cases}
			R_{ij}(x-y),                                                                            & 1\le i,j\le d;       \\
			-\sum_{k=1}^d \frac{\partial R_{ik}}{\partial x_k}(x-y)=-(\nabla\cdot R)_i(x-y), & 1\le i\le d,\ j=d+1; \\
			\sum_{k=1}^d \frac{\partial R_{kj}}{\partial x_k}(x-y),                                 & i=d+1,\ 1\le j\le d; \\
			-\sum_{k,\ell=1}^d \frac{\partial^2 R_{k\ell}}{\partial x_k\partial x_\ell}(x-y),       & i=j=d+1.             \\
		\end{cases}
	\end{equation}
	We also have the auxiliary functions
	\begin{align*}
		\kunitanotation{C}^j(x,t) & = \sum_{i=1}^d \frac{\partial \kunitanotation{A}^{ij}}{\partial \kunitanotation{y}^i}(x,y,t)|_{y=x} = -\sum_{i=1}^d \frac{\partial R_{ij}}{\partial x_i}(0)=(\nabla\cdot R)_j(0) \\
		\kunitanotation{D}(x,t)   & = \sum_{i=1}^d \frac{\partial \kunitanotation{A}^{i,d+1}}{\partial y^i}(x,y,t)|_{y=x}=\sum_{k,\ell=1}^d\frac{\partial R_{k\ell}}{\partial x_k\partial x_\ell}(0).
	\end{align*}
	Thus,  \cite[p.~302, (3)]{Kun97} becomes in our setting
	\begin{align*}
		\kunitanotation{\tilde{L}}u & =\frac12\sum_{i,j=1}^d R_{ij}(0)\frac{\partial^2u}{\partial x_i\partial x_j}+\sum_{i=1}^d \left(-(\nabla\cdot R)_i +\frac{1}{2}(\nabla\cdot R)_i \right)\frac{\partial u}{\partial x_i} \\%
		                            & =\frac12\sum_{i,j=1}^d R_{ij}(0)\frac{\partial^2u}{\partial x_i\partial x_j}-\frac12 (\nabla\cdot R)\cdot\nabla u,                                                                      %
	\end{align*}
	and so by \cite[p.~302, (4)]{Kun97}, we have
	\begin{align}
		\mu(t) & = \mu(0)+\int_0^t \left[\frac12\nu\Delta \mu(s)+\frac{1}{2}(\nabla\cdot R)(0)\cdot \nabla u(s)+\kunitanotation{\tilde{L}}u(s)\right]\,\dif s -\int_0^t \nabla \mu(t)\cdot V(\dif s)\notag \\&\qquad-\int_0^t\mu(s)(\nabla\cdot V)(\dif s)\notag\\
		       & = \mu(0)+\int_0^t \frac12\tr[(\nu I_d +R(0))\nabla^2 \mu(s)]\,\dif s -\int_0^t \nabla \mu(t)\cdot V(\dif s)-\int_0^t\mu(s)(\nabla\cdot V)(\dif s).
	\end{align}
	Thus $u$ satisfies the Itô SPDE \cref{eq:utdef}.
\end{proof}

\subsection{The second-moment PDE\label{subsec:second-moment-PDE}}

As described in the introduction, we now want to write a PDE for the
second moments of $u(t)$. In this section, we do this using the Itô formula, and then perform a change of variables to simplify the resulting PDE. We first consider
\begin{equation}\label{e.defu2}
	u_2(t,x,y)=u(t,x)u(t,y).
\end{equation}
Since $u(t,\cdot)$ is the quenched density of $X(t)$, we know that if $u(0)$ is a delta, then $u_2(t,x,y)$ is the joint quenched density of $(X(t),Y(t))$ with
\begin{equation}\label{eq:defXYt}
	\begin{aligned}
		 & X(t)=X(0)+B_1(t)+\int_{0}^{t} V(\dif s,X(s)), \\
		 & Y(t)=Y(0)+B_2(t)+\int_0^t V(\dif s, Y(s)),
	\end{aligned}
\end{equation}
where $B_1,B_2$ are independent Brownian motions that are also independent from $V$. Thus, $u_2$ encodes the correlation of the two passive scalars in the same random environment.
From the Itô formula and the SPDE \cref{eq:utdef} (or, in the case when $u(0)$ is a delta, by redoing the
computation in \cref{prop:utsolvesSPDE} but for a flow on $\mathbf{R}^{2d}$
where the first and last $d$ coordinates are forced by the same instance of
$V$ but two independent Brownian motions $B_1$ and $B_2$), we see that $u_2$
satisfies the SPDE
\begin{equation}
	\begin{aligned}\dif u_2(t,x,y) & =\frac{1}{2}\tr[(\nu I_{d}+R(0))^{\otimes2}\nabla^{2}u_2](t,x,y)\dif t-u(t,x)\nabla\cdot[u(t,y) V(\dif t,y)]-u(t,y)\nabla\cdot[u(t,x)V(\dif t,x)] \\
                               & \qquad+ \sum_{i,j=1}^{d}\frac{\partial^{2}}{\partial x_{i}\partial y_{j}}(u_2(t,x,y)R_{ij}(x-y))\dif t,
	\end{aligned}
	\label{eq:utotimes2}
\end{equation}
again in the sense of \cref{eq:generalizedsolution}. %
If we define
\[
	Q_{t}(x,y)=\mathbf{E}u_2(t,x,y),
\]
then $Q_{t}$ lives in polynomially-weighted Sobolev
space by \cite[Lemma 3.1]{Kun94gsaspde}. By definition, $Q_t$ is the annealed density of $(X(t),Y(t))$ defined in \cref{eq:defXYt}. Now we take expectations
in \cref{eq:utotimes2}. Rigorously, this could be done by using \cite[Theorem 3.1]{Kun94gsaspde}
again, but this time taking conditional expectation with respect to
the null filtration. In this way, we see that $Q_{t}$, considered
as a tempered distribution, is the unique solution to the PDE
\begin{align}
	\partial_{t}Q_{t}(x,y) & =\frac{1}{2}\tr[\nabla^{2}[(\nu I_{d}+R(0))^{\otimes2}Q_{t}]](x,y)+ \sum_{i,j=1}^{d}\frac{\partial^{2}}{\partial x_{i}\partial y_{j}}(Q_{t}(x,y)R_{ij}(x-y)),\qquad t>0;\label{eq:QPDE} \\
	Q_{0}(x,y)             & =u_0(x)u_0(y)\label{eq:Qic}
\end{align}
in the ``generalized'' sense of \cite[(2.1)]{Kun94gsaspde} (which
means that the corresponding integral equation holds when $Q_{t}$
is integrated against a Schwartz test function).

Now we make change of variables
\begin{equation}
	x\mapsto w+z/2,\qquad y\mapsto w-z/2,\label{eq:xychgvar}
\end{equation}
and put
\[
	S_{t}(w,z)=Q_{t}(w+z/2,w-z/2).
\]
With $X(t),Y(t)$ defined in \cref{eq:defXYt}, we further define the center of mass and the relative distance by \begin{equation}\label{eq:WZdefs}
	W(t)=(X(t)+Y(t))/2, \quad\quad Z(t)=X(t)-Y(t),
\end{equation} so $S_t(w,z)$ is the annealed  density of $(W(t),Z(t))$.
Define the matrix $A(z)$ by
\begin{equation}
	\begin{aligned}A(z) & =\begin{pmatrix}A_{11} & A_{12} \\
               A_{21} & A_{22}
		\end{pmatrix}(z)=\frac{1}{2}\begin{pmatrix}\frac{1}{2}I_{d} & \frac{1}{2}I_{d} \\
               I_{d}            & -I_{d}
		\end{pmatrix}\left[\nu I_{2d}+\begin{pmatrix}R(0)      & R(z)^{\mathrm{T}} \\
               R(z) & R(0)
			\end{pmatrix}\right]\begin{pmatrix}\frac{1}{2}I_{d} & I_{d}  \\
               \frac{1}{2}I_{d} & -I_{d}
		\end{pmatrix} \\
                    & =\begin{pmatrix}\frac{1}{4}\left[\nu I_{d}+R(0)\right]+\frac{1}{8}\left[R(z)+R(z)^{\mathrm{T}}\right] & \frac{1}{4}\left[R(z)^{\mathrm{T}}-R(z)\right]  \\
               \frac{1}{4}\left[R(z)-R(z)^{\mathrm{T}}\right]                                   & \nu I_{d}+R(0)-\frac12\left[R(z)+R(z)^{\mathrm{T}}\right]
		\end{pmatrix}.
	\end{aligned}
	\label{eq:Adef}
\end{equation}
Then from \cref{eq:QPDE} we obtain
\begin{align}
	\partial_{t}S_{t}(w,z) & =\mathcal{L}^{*}S_{t}(w,z)=\tr\left[\nabla^{2}(AS_{t})(w,z)\right];\label{eq:S-PDE} \\
	S_{0}(w,z)             & =u_{0}(w+z/2)u_{0}(w-z/2),\label{eq:Sic}
\end{align}
where we have defined the differential operator
\begin{equation}
	\mathcal{L}f(w,z)=\tr[A(z)\nabla^{2}f(w,z)],
	\label{eq:Ldef}
\end{equation}
and $\mathcal{L}^{*}$ its adjoint
\[
	\mathcal{L}^{*}f(w,z)=\tr[\nabla^{2}(Af)(w,z)],
\]
where we use the notation, if $A=(a_{ij})$,
\[
	\tr[\nabla^{2}(Af)]=\sum_{i,j=1}^{d}\frac{\partial^{2}}{\partial x_{i}\partial x_{j}}[a_{ij}f].
\]
We emphasize that \cref{eq:S-PDE}--\cref{eq:Sic} is simply a deterministic
change of variables from \cref{eq:QPDE}--\cref{eq:Qic}. Alternatively, one could start from \cref{eq:defXYt} to write down the equation satisfied by $(W(t),Z(t))$, then derive the PDE  satisfied by its annealed density, which is \cref{eq:S-PDE}.

If $X:\mathbf{R}^{d}\to\mathbf{R}^{d}$ is a stationary Gaussian random
field with correlation function $\mathbf{E}[X(z)X(0)^{\mathrm{T}}]=R(z)$,
then we have
\[
	\mathbf{E}\left[\begin{pmatrix}X(z) \\
		X(0)
	\end{pmatrix}\begin{pmatrix}X(z)^{\mathrm{T}} & X(0)^{\mathrm{T}}\end{pmatrix}\right]=\begin{pmatrix}R(0)  & R(z) \\
               R(-z) & R(0)
	\end{pmatrix}=\begin{pmatrix}R(0)              & R(z) \\
               R(z)^{\mathrm{T}} & R(0)
	\end{pmatrix},
\]
so the matrix on the right is nonnegative-definite, and thus from \cref{eq:Adef} we conclude that $A(z)$
is positive-definite uniformly over all $z\in\mathbf{R}^{d}$. By the assumption  $R\in \mathcal{C}_{\mathrm{c}}^{\infty}$,  $A(z)$ is also smooth in $z$.

Now by the theory of parabolic PDEs (which relies on the ellipticity
of $A$; see e.g.~\cite[§1.6]{Fri64}), we know that the PDE \cref{eq:S-PDE}
has a fundamental solution. Thus, \cref{eq:S-PDE}--\cref{eq:Sic} has
a classical solution given by integration of the initial measure against
the fundamental solution. Since it is clear that this classical solution
is also a tempered distribution and satisfies \cref{eq:S-PDE}--\cref{eq:Sic}
in the ``generalized'' sense of Kunita \cite{Kun94gsaspde}, for
which there is a uniqueness statement, the function $S_{t}$ in fact
is given by integration of the initial condition \cref{eq:Sic} against
the fundamental solution. In the sequel, we mean this solution when
we talk about ``the'' solution to \cref{eq:S-PDE}--\cref{eq:Sic}.
(Any other solution must have extremely fast growth
as $|x|\to\infty$.)

\subsubsection{Bounds on the fundamental solution\label{subsec:PDEbounds}}

For notational convenience, we will often write $\omega=(w,z)$. Let
$\Gamma_{t}$ be the fundamental solution for \cref{eq:S-PDE}, so that
the solution to \cref{eq:S-PDE} satisfies
\[
	S_{t}(\omega)=\int\Gamma_{t-s}(\omega;\omega')S_{s}(\omega')\,\dif\omega'
\]
for $s<t$ and $\omega\in\mathbf{R}^{2d}$. We note that $\Gamma_{t}$
is the fundamental solution for the non-divergence form parabolic
PDE
\begin{equation}
	\partial_{t}g=\mathcal{L}g,
	\label{eq:nondivergenceformPDE}
\end{equation}
with its arguments swapped, i.e.,
\[
	g_{t}(\omega)=\int\Gamma_{t-s}(\omega';\omega)g_{s}(\omega')\,\dif\omega'.
\]
In this section we will prove some bounds
on $\Gamma_{t}$ using tools from the theory of parabolic PDE, in particular the bounds
on the fundamental solutions of nondivergence-form parabolic PDEs given in \cite{Esc00}.

Recall from \cref{eq:Adef} that
\begin{equation}\label{eq:defA22}
	A_{22}(z)=\nu I_{d}+R(0)-\frac12\left[R(z)+R(z)^{\mathrm{T}}\right].
\end{equation}
We first need the
following proposition, which will also be useful later.
\begin{prop}
	\label{prop:bexists}There is a unique function $\chi\in\mathcal{C}^{\infty}(\mathbf{R}^{d};\R)$
	and a constant $C<\infty$ so that
	\begin{equation}
		\tr[\nabla^{2}(A_{22}\chi)]\equiv0,\label{eq:bsolves-1}
	\end{equation}
	\begin{equation}
		\chi-1\in L^{p}(\mathbf{R}^{d})\qquad\text{for any }p>1,\label{eq:babout1-inLp}
	\end{equation}
	\begin{equation}
		C^{-1}\le\inf_{\mathbf{R}^{d}}\chi\le\sup_{\mathbf{R}^{d}}\chi\le C,\label{eq:bboundedaboveandawayfromzero-1}
	\end{equation}
	and
	\begin{equation}
	 |\chi(x)-1|\le C|x|^{-d}\text{ for all $x$ with $|x|\ge 1$.}\label{eq:powerlaw}
	\end{equation}

\end{prop}

\begin{rem}
	\label{rem:easychi}In the case of $R(z)=f(z)I_{d}$ for some scalar function $f\in \mathcal{C}^{\infty}_c(\mathbf{R}^{d};\R)$ (which is always the case in $d=1$), we can take
	\[
		\chi(z)=\frac{\nu+f(0)}{\nu+f(0)-f(z)}
	\]
	which evidently satisfies \cref{eq:bsolves-1}--\cref{eq:bboundedaboveandawayfromzero-1}.
	In fact, it satisfies \cref{eq:babout1-inLp} with $p=1$ as well.
\end{rem}
 
\begin{rem}\label{rem:incompressiblechi}
 In the case when $\sum_{i=1}^d \frac{\partial R_{ij}}{\partial x_i}\equiv 0$ for each $j$ (i.e. when $V$ is incompressible almost surely), it is clear from \cref{eq:defA22} that  $\chi\equiv 1$.
\end{rem}

\begin{rem}From \cref{eq:defA22} and the fact that $R$ is compactly supported, one can view $A_{22}$ as a perturbation of the constant matrix $\nu I_d+R(0)$. Since \cref{eq:bsolves-1} is the equation for the invariant measure of the process $Z(t)$ defined in \cref{eq:WZdefs}, \cref{prop:bexists} is essentially to quantify the fact that the invariant measure is a perturbation of the Lebesgue measure.
\end{rem}

\begin{proof}[Proof of \cref{prop:bexists}]
	By \cref{rem:easychi}, we can assume that $d\ge 2$, so we can use the results of \cite{Esc00}.
	Since $A$ is uniformly positive definite, Theorem~1.1 of \cite{Esc00}
	implies that there is a unique, up to a scalar multiple normalization,
	$\chi:\mathbf{R}^{d}\to\mathbf{R}_{\ge0}$ satisfying \cref{eq:bsolves-1}
	in a weak sense. Using the assumption that $R$ is smooth, \cite[Theorem 1.4.6]{BKRS15}
	ensures that $\chi$ is smooth as well. Therefore, $\chi$ in fact satisfies
	\cref{eq:bsolves-1} in a classical sense.

	Now we need to prove \cref{eq:babout1-inLp} and \cref{eq:bboundedaboveandawayfromzero-1}.
	Our approach is based on the proof of \cite[Theorem 1.5]{Esc00},
	the difference being that we make stronger assumptions and obtain
	stronger results. For the purpose of this proof only, we make a deterministic,
	linear change of coordinates so that we can assume that $\nu I_{d}+R(0)=I_{d}$.
	This does not affect the conclusions of the proposition (up to the
	choice of constants). This means that $A_{22}(z)=I_{d}+E(z)$, where
	$E$ is compactly-supported, say on $B_{M}(0)$ for some $M>0$. Throughout the proof, to simplify the notation we  write $\sum_{i,j}=\sum_{i,j=1}^d$ and $\A=A_{22}$. Now
	define
	\[
		f_{x}(r)=\int_{B_{r}(x)}\chi(z)\,\dif z.
	\]
	Then we claim that
	\begin{equation}\label{eq:intebyP}
		r\int_{B_{r}(x)}\chi(z)\tr \A(z)\,\dif z=\int_{\partial B_{r}(x)}\chi(z)\A(z)(z-x)\cdot(z-x)\,\dif\mathcal{H}^{d-1}(z),
	\end{equation}
	where $\dif \mathcal{H}^{d-1}$ is the surface measure. To show \cref{eq:intebyP}, we write
	\[
		\begin{aligned}
			r\int_{B_{r}(x)}\chi(z)\tr \A(z)\,\dif z & =r\sum_{i,j} \int_{B_r(x)} \chi(z)\A_{ij}(z)\delta_{ij} \dif z                                                          \\
			                                                & =\frac{r}{2}\sum_{i,j} \int_{B_r(x)} \chi(z)\A_{ij}(z) \partial_{z_iz_j} (|z-x|^2-r^2) \dif z                           \\
			                                                & =\frac{r}{2}\sum_{i,j} \int_{B_r(x)} \partial_{z_i} \bigg(\chi(z)\A_{ij}(z)\partial_{z_j}(|z-x|^2-r^2)\bigg) \dif z     \\
			                                                & \qquad-\frac{r}{2}\sum_{i,j}\int_{B_r(x)}\partial_{z_i}\bigg(\chi(z)\A_{ij}(z)\bigg) \partial_{z_j}(|z-x|^2-r^2) \dif z \\&\eqqcolon I_1-I_2.
		\end{aligned}
	\]
	For $I_1$,  we apply the divergence theorem to see that it is equal to the r.h.s. of \cref{eq:intebyP}.
	For $I_2$, by the fact that $\tr[\nabla^{2}(\A\chi)]\equiv0$, we have
	\[
		\begin{aligned}
			I_2=\frac{r}{2}\sum_{i,j} \int_{B_r(x)} \partial_{z_iz_j} \bigg(\chi(z)\A_{ij}(z)(|z-x|^2-r^2)\bigg) \dif z=0
		\end{aligned}
	\]
	where the last identity comes from another application of divergence theorem. So \cref{eq:intebyP} is proved.

	Now we have
	\begin{align}
		f_{x}'(r) & =\int_{\partial B_{r}(x)}\chi(z)\,\dif\mathcal{H}^{d-1}(z)=\frac{1}{r^{2}}\int_{\partial B_{r}(x)}|z-x|^{2}\chi(z)\,\dif\mathcal{H}^{d-1}(z)\nonumber \\
		          & =\frac{1}{r^{2}}\int_{\partial B_{r}(x)}\chi(z)\A(z)(z-x)\cdot(z-x)\,\dif\mathcal{H}^{d-1}(z)-D_{1}(r,x)\nonumber                                     \\
		          & =\frac{1}{r}\int_{B_{r}(x)}\chi(z)\tr \A(z)\,\dif z-D_{1}(r,x)=\frac{d}{r}\int_{B_{r}(x)}\chi(z)\,\dif z+D_{2}(r,x)-D_{1}(r,x)\nonumber               \\
		          & =\frac{d}{r}f_{x}(r)+D_{2}(r,x)-D_{1}(r,x).\label{eq:fDE}
	\end{align}
	where we used the fact that $\A(z)=I_d+E(z)$ and we defined
	\begin{align*}
		D_{1}(r,x) & =\frac{1}{r^{2}}\int_{\partial B_{r}(x)}\chi(z)E(z)(z-x)\cdot(z-x)\,\dif\mathcal{H}^{d-1}(z)               \\
		           & =\frac{1}{r^{2}}\int_{\partial B_{r}(x)\cap B_{M}(0)}\chi(z)E(z)(z-x)\cdot(z-x)\,\dif\mathcal{H}^{d-1}(z).
	\end{align*}
	and
	\[
		D_{2}(r,x)=\frac{1}{r}\int_{B_{r}(x)\cap B_{M}(0)}\chi(z)\tr E(z)\,\dif z.
	\]
	We note that there is a constant $D$, independent of $x$ and $r$,
	so that
	\begin{equation}
		|D_{1}|,r|D_{2}(r,x)|\le D,\qquad r>0,x\in\mathbf{R}^{d}.\label{eq:D1D2bd}
	\end{equation}

	First we consider the case $x=0$. We note that, whenever $r\geq M$, we have  $D_{1}(r,0)=0$  and $D_{2}(r,0)=\frac{M}{r}D_{2}(M,0)$.
	Therefore, we have for $r\ge M$ that
	\[
		f_{0}'(r)=\frac{d}{r}\left[f_{x}(r)+d^{-1}MD_{2}(M,0)\right],
	\]
	so, solving the ODE, we obtain for $r\ge M$ that
	\begin{equation}
		f_{0}(r)=\kappa r^{d}-d^{-1}MD_{2}(M,0)\label{eq:solveforf0}
	\end{equation}
	for some constant $\kappa$.
	We fix the normalization of $\chi$ so that $\kappa$ is the volume
	of the unit ball in $\mathbf{R}^{d}$. In other words, for $r\gg1$, we have
	\[
		\int_{B_r(0)} \chi(z)\,\dif z=f_0(r)\sim \int_{B_r(0)}\1(z)\,\dif z.
	\]

	Now we consider general $x$. From \cref{eq:fDE} we have
	\begin{equation}
		f_{x}'(r)=\frac{d}{r}f_{x}(r)+D_{2}(r,x)-D_{1}(r,x).\label{eq:fxprime}
	\end{equation}
	We note that $D_{1}(r,x)=0$ and $D_{2}(r,x)=\frac{|x|+M}{r}D_{2}(|x|+M,x)$
	whenever $r\ge|x|+M$. Therefore, we have
	\[
		f_{x}'(r)=\frac{d}{r}\left[f_{x}(r)+d^{-1}(|x|+M)D_{2}(|x|+M,x)\right],\qquad r\ge|x|+M,
	\]
	and thus
	\begin{equation}
		f_{x}(r)=\kappa_{x}r^{d}-d^{-1}(|x|+M)D_{2}(|x|+M,x),\qquad r\ge|x|+M,\label{eq:solveforf0-1}
	\end{equation}
	for some constant $\kappa_{x}$. But since $f_{0}(r-|x|)\le f_{x}(r)\le f_{0}(r+|x|)$
	for all $r\ge|x|$, comparing \cref{eq:solveforf0} and \cref{eq:solveforf0-1}
	and taking $r$ large we see that in fact we must have $\kappa_{x}=\kappa$
	for all $x$, and thus
	\begin{equation}
		f_{x}(r)=\kappa r^{d}-d^{-1}(|x|+M)D_{2}(|x|+M,x),\qquad r\ge|x|+M.\label{eq:fxrwhenrislarge}
	\end{equation}
	Using \cref{eq:D1D2bd} in \cref{eq:fxrwhenrislarge}, we have that
	\begin{equation}
		|f_{x}(r)-\kappa r^{d}|\le d^{-1}D,\qquad r\ge|x|+M.\label{eq:fxrclosetokappard}
	\end{equation}

	Now assume that $|x|\ge M+1$. We have from \cref{eq:fxprime} and \cref{eq:D1D2bd}
	that, as long as $r\ge|x|-M$, we have $r\ge1$ and hence
	\[
		\left|\frac{\dif}{\dif r}\left(r^{-d}f_{x}(r)\right)\right|\le Dr^{-d}\left[1+1/r\right]\le2Dr^{-d},
	\]
	so
	\begin{equation}
		\left|(|x|+M)^{-d}f_{x}(|x|+|M|)-(|x|-M|)^{-d}f_{x}(|x|-M)\right|\le4DM(|x|-M)^{-d}.\label{eq:slippasttheball}
	\end{equation}
	Since $\chi$ is harmonic in $B_{|x|-M}(x)$ (recall that $\A(z)=I_d$ for $|z|\geq M$), we have
	\begin{equation}
		(|x|-M|)^{-d}f_{x}(|x|-M)=\kappa\chi(x).\label{eq:harmonic}
	\end{equation}
	Using \cref{eq:fxrclosetokappard} and \cref{eq:harmonic} in \cref{eq:slippasttheball},
	we have
	\begin{equation}
		\kappa\left|1-\chi(x)\right|\le d^{-1}D(|x|+M)^{-d}+4DM(|x|-M)^{-d}.\label{eq:1minusbx}
	\end{equation}

	Now since $\chi$ is smooth, \cref{eq:1minusbx} implies \cref{eq:babout1-inLp}.
	Also, \cref{eq:1minusbx} implies \cref{eq:bboundedaboveandawayfromzero-1} and \cref{eq:powerlaw}
	for large $|x|$. But for $x$ in any bounded domain, \cref{eq:bboundedaboveandawayfromzero-1}
	holds by the smoothness of $\chi$ and $E$ and the strong maximum
	principle. This completes the proof.
\end{proof}
\begin{lem}
	\label{lem:fundsolnbound}There exists a constant $C=C(\nu,R)<\infty$
	so that, for all $t>0$ and all $\omega,\omega'\in\mathbf{R}^{2d}$,
	we have
	\begin{equation}
		\Gamma_{t}(\omega;\omega')\le Ct^{-d}\exp\left\{ -C^{-1}t^{-1}|\omega-\omega'|^{2}\right\} .\label{eq:Gammatbd}
	\end{equation}
\end{lem}

\begin{proof}
	Recall that we write $\omega=(w,z)$ and $\omega'=(w',z')$. Let %
	\[
		\tilde{\chi}(\omega)=\tilde{\chi}(w,z)=\frac{\chi(z)}{\int_{|\omega'|\leq 1} \chi(z')\,\dif w'\dif z'},
	\]
	with $\chi$ as in \cref{prop:bexists}, which satisfies $\mathcal{L}^{*}\tilde{\chi}=0$
	by \cref{eq:bsolves-1}, and also satisfies $\int_{|\omega|^2\leq 1}\tilde{\chi}(\omega)\,\dif\omega=1$.
	By \cref{eq:bboundedaboveandawayfromzero-1}, there is a constant $C=C(\nu,R)<\infty$
	so that for all $t>0$ and all $\omega\in\mathbf{R}^{2d}$, we have
	\[
		\tilde{\chi}(\omega)\le C
	\]
	and
	\[
		\int_{B_{t}(\omega)}\tilde{\chi}(\omega')\,\dif\omega'\ge C^{-1}t^{2d}.
	\]
	Using these bounds in the result of \cite[Theorem 1.2]{Esc00} (noting
	that our $\tilde{\chi}$ is denoted there by $W$), we have another
	constant $C=C(\nu,R)$ so that, for all $t>0$ and all $\omega,\omega'\in\mathbf{R}^{2d}$,
	the estimate \cref{eq:Gammatbd} holds. Note that \cite{Esc00} is written
	in terms of the nondivergence form PDE \cref{eq:nondivergenceformPDE},
	but the fundamental solutions are related by simply swapping the arguments
	and so the same bound holds for $\Gamma_{t}$. The proof is complete.
	\end{proof}
	\begin{rem}
	Note that \cite{Esc00} assumes that the dimension $d$ is at least $2$, but the proof of the upper bound in \cite[Theorem~1.2]{Esc00} given there works also for $d=1$. Actually, the proof is in fact simpler as it follows just from the Krylov--Safonov Harnack inequality \cite[Theorem~3.1]{Esc00} and the construction of a subsolution \cite[Lemma~3.1]{Esc00} as in the derivation leading to \cite[(3.8)]{Esc00}, using the explicit construction of the invariant measure given in \cref{rem:easychi} and the fact that in $d=1$, what \cite{Esc00} calls a ``normalized adjoint solution'' is in fact just a solution to the original nondivergence-form equation.
\end{rem}

\begin{lem}
	\label{lem:fundsolnderivbound}
	There exists  a constant $C=C(\nu,R)<\infty$
	so that, for all $t\ge1$ and all $\omega,\omega'\in\mathbf{R}^{2d}$, we have
	\begin{equation}
		\left|\nabla_{\omega'}^{2}\Gamma_{t}(\omega;\omega')\right|\le Ct^{-d}\exp\left\{ -C^{-1}t^{-1}|\omega-\omega'|^{2}\right\} .\label{eq:GammaHessBd}
	\end{equation}
\end{lem}
\begin{proof}
	By the Chapman--Kolmogorov equation we have
	\begin{equation*}
		\Gamma_t(\omega;\omega')=\int \Gamma_{t-1/2}(\omega;\omega'')\Gamma_{1/2}(\omega'';\omega')\,\dif\omega''.
	\end{equation*}
	Thus we have
	\begin{equation}\label{eq:splitintegral}
		|\nabla_{\omega'}^2\Gamma_t(\omega;\omega')|\le\int |\Gamma_{t-1/2}(\omega;\omega'')|\cdot|\nabla_{\omega'}^2\Gamma_{1/2}(\omega'';\omega')|\,\dif\omega''.
	\end{equation}
	By \cite[Theorem~9.6.7 on p.~261]{Fri64} (which again concerns the fundamental solution for the adjoint problem \cref{eq:nondivergenceformPDE}, but that corresponds to our fundamental solution by swapping the arguments), using the assumed smoothness of $A$, we have a constant $C<\infty$ so that
	\[
		|\nabla_{\omega'}^2\Gamma_{1/2}(\omega'';\omega')|\le C\exp\{-C^{-1}|\omega''-\omega|^2\}.
	\]
	Using this bound along with \cref{lem:fundsolnbound} in \cref{eq:splitintegral} we obtain \cref{eq:GammaHessBd}.
\end{proof}

\section{The stationary solution\label{sec:stationary}}
In this section we show the existence of a spacetime-stationary solution for the SPDE \cref{eq:utdef}, and some properties of the spacetime-stationary solution. The strategy is to consider the solutions to \cref{eq:utdef} started at large negative times, and show that the resulting sequence of fields is a Cauchy sequence in $L^2$.

Let $u^{[M]}$ solve \cref{eq:utdef} but with constant initial condition
$1$ at time $-M$, i.e.,
\begin{equation}\label{eq:uM}
	\begin{aligned}
		 & \dif u^{[M]}( t)  =\frac{1}{2}\tr[(\nu I_{d}+R(0))\nabla^{2}u^{[M]}(t)]\dif t-\nabla\cdot[u^{[M]}(t)V(\dif t)], \quad\quad t>-M \\
		 & u^{[M]}(-M)\equiv1.
	\end{aligned}
\end{equation}
The main result of this section is the following
proposition.
\begin{prop}
	\label{prop:uMCauchy}For any $t\in\mathbf{R}$, $x\in\mathbf{R}^{d}$,
	the sequence $(u_{t}^{[M]}(x))_{M>-t}$ is a Cauchy sequence in
	$L^{2}(\Omega)$. In particular, for any $\eps>0$ (or $\eps=0$ if
	$d=1$) there exists a constant $C=C(\eps,\nu,R)<\infty$ so that for any $M_1,M_2>-t$,
	\begin{equation}
		\mathbf{E}\left|u^{[M_{1}]}(t,x)-u^{[M_{2}]}(t,x)\right|^{2}\le C[(t+M_{1})^{-(d-\eps)/2}+(t+M_{2})^{-(d-\eps)/2}].\label{eq:Cauchyquantitative}
	\end{equation}
\end{prop}

\begin{proof}
	Define
	\[
		S_{t}^{[M]}(z)=\mathbf{E}u^{[M]}(t,0)u^{[M]}(t,z),
	\]
	which is equal to $\mathbf{E}u^{[M]}(t,w-z/2)u^{[M]}(t,w+z/2)$ for
	all $w\in\mathbf{R}^{d}$, due to the fact that the noise is spatially translation-invariant and that the initial data is constant.
	Thus $S_{t}^{[M]}$ satisfies the PDE
	\begin{align}
		\partial_{t}S_{t}^{[M]}(z) & =\overline{\mathcal{L}}^{*}S_{t}^{[M]}(z)=\frac{1}{2}\tr[\nabla^{2}(A_{22}S_{t}^{[M]})(z)],\label{eq:SMPDE} \\
		S_{-M}^{[M]}(z)            & =1,\label{eq:SMic}
	\end{align}
	where we have defined $\overline{\mathcal{L}}f(z)=\frac{1}{2}\tr[A_{22}(z)\nabla^{2}f(z)].$
	The problem \cref{eq:SMPDE}--\cref{eq:SMic} is obtained from \cref{eq:S-PDE}--\cref{eq:Sic}
	by using the space-translation-invariance and the fact that the initial data is a constant. The PDE \cref{eq:SMPDE}
	has fundamental solution $\overline{\Gamma}$ given by
	\[
		\overline{\Gamma}_{t}(z;z')=\int\Gamma(y,z;y',z')\,\dif y,
	\]
	which in fact is independent of $y'$. Integrating \cref{eq:Gammatbd}
	over $y$, we have a constant $C=C(\nu,R)<\infty$ so that
	\begin{equation}
		\overline{\Gamma}_{t}(z;z')\le Ct^{-d/2}\exp\left\{ -C^{-1}t^{-1}|z-z'|^{2}\right\} .\label{eq:gammabarbound}
	\end{equation}
	Now we recall the function $\chi$ from \cref{prop:bexists}. We note
	that
	\begin{equation}
		\int\overline{\Gamma}_{t}(z;z')\chi(z')\,\dif z'=\chi(z)\label{eq:chipreserved}
	\end{equation}
	by \cref{eq:bsolves-1}. Thus we have
	\begin{align}
		S_{t}^{[M]}(z)=\int\overline{\Gamma}_{t+M}(z;z')\,\dif z' & =\int\overline{\Gamma}_{t+M}(z;z')\left(\chi(z')-[\chi(z')-1]\right)\,\dif z'\nonumber \\
		                                                          & =\chi(z)-\int\overline{\Gamma}_{t+M}(z;z')[\chi(z')-1]\,\dif z'.\label{eq:SMtclosetob}
	\end{align}
	By Hölder's inquality, for $1/p+1/q=1$ we have
	\begin{equation}
		\left|\int\overline{\Gamma}_{t+M}(z;z')[\chi(z')-1]\,\dif z'\right|\le\|\overline{\Gamma}_{t+M}(z;\cdot)\|_{L^{q}(\mathbf{R}^{d})}\|\chi-1\|_{L^{p}(\mathbf{R}^{d})}.
		\label{eq:applyholder}
	\end{equation}
	By \cref{eq:gammabarbound}, we have
	\begin{equation}
		\begin{aligned}\|\overline{\Gamma}_{t}(z;\cdot)\|_{L^{q}(\mathbf{R}^{d})} & \le\left(\int Ct^{-dq/2}\exp\left\{ -qC^{-1}t^{-1}|z-z'|^{2}\right\} \,\dif z'\right)^{1/q}                  \\
                                                                          & =\left(\int Ct^{-d(q-1)/2}\exp\left\{ -qC^{-1}|z'|^{2}\right\} \,\dif z'\right)^{1/q}\le Ct^{-\frac{d}{2p}}.
		\end{aligned}
		\label{eq:GammabarLpbound}
	\end{equation}
	Thus, using \cref{eq:babout1-inLp} and \cref{eq:GammabarLpbound} in
	\cref{eq:applyholder} and then substituting into \cref{eq:SMtclosetob},
	we have for any $\eps>0$ (choosing $p=d/(d-\eps)$), there exists a constant $C$ so that
	\begin{equation}
		|S_{t}^{[M]}(z)-\chi(z)|\le C(t+M)^{-(d-\eps)/2}.\label{eq:StMconverges}
	\end{equation}
	When $d=1$, by \cref{rem:easychi} we can take $p=1$, $q=\infty$,
	and thus $\eps=0$.

	Now define, for $M_{1}<M_{2}$, %
	\[
		S_{t}^{[M_{1},M_{2}]}(z)=\mathbf{E}(u^{[M_{1}]}(t,0)-u^{[M_{2}]}(t,0))(u^{[M_{1}]}(t,z)-u^{[M_{2}]}(t,z)), \quad\quad t>-M_1,  z\in\R^d.
	\]
	Then $S_{t}^{[M_{1},M_{2}]}$ again satisfies the PDE \cref{eq:SMPDE},
	since $u^{[M_{1}]}-u^{[M_{2}]}$ also satisfies \cref{eq:utdef} by
	linearity. On the other hand, we have the corresponding initial condition
	\begin{align*}
		S_{-M_{1}}^{[M_{1},M_{2}]}(z) & =\mathbf{E}[(u^{[M_{1}]}(-M_1,0)-u^{[M_{2}]}(-M_1,0))(u^{[M_{1}]}(-M_1,z)-u^{[M_{2}]}(-M_1,z))] \\
		                              & =\mathbf{E}[(1-u^{[M_{2}]}(-M_1,0))(1-u^{[M_{2}]}(-M_1,z))]                                     \\
		                              & =\mathbf{E}[u^{[M_{2}]}(-M_1,0)u^{[M_{2}]}(-M_1,z)]-1                                           \\
		                              & =S_{-M_{1}}^{[M_{2}]}(z)-1.
	\end{align*}
	Here we used the fact that $\mathbf{E}u^{[M_2]}\equiv1$. From this and the linearity of \cref{eq:SMPDE} we further conclude that
	\begin{equation}
		S_{t}^{[M_{1},M_{2}]}(z)=S_{t}^{[M_{2}]}(z)-S_{t}^{[M_{1}]}(z), \quad\quad t>-M_1,z\in\R^d.
		\label{eq:StMisdifference}
	\end{equation}
	Combining \cref{eq:StMconverges}, \cref{eq:StMisdifference}, and the
	triangle inequality, we have
	\[
		\mathbf{E}\left|u^{[M_{1}]}(t,x)-u^{[M_{2}]}(t,x)\right|^{2}=S_{t}^{[M_{1},M_{2}]}(0)\le C[(t+M_{1})^{-(d-\eps)/2}+(t+M_{2})^{-(d-\eps)/2}],
	\]
	which is \cref{eq:Cauchyquantitative}.
\end{proof}
\begin{cor}
	\label{cor:Uexists}There is a positive random function $U:\mathbf{R}\times\mathbf{R}^{d}\to\mathbf{R}$
	so that, for every $t\in\mathbf{R}$ and $x\in\mathbf{R}^{d}$, we
	have for any $\eps>0$ (or $\eps=0$ if $d=1$) that
	\[
		\lim_{M\to\infty}(t+M)^{(d-\eps)/2}\mathbf{E}|u^{[M]}(t,x)-U(t,x)|^{2}=0.
	\]
	Moreover, $U$ is stationary under time and space translations, $\mathbf{E}  U\equiv 1$, and
	$U$ solves the SPDE \cref{eq:utdef}. Finally, for $x_1,x_2\in\R$, we have
	\begin{equation}\mathbf{E} [U(t,x_1) U(t,x_2)]=\chi(x_1-x_2),\label{eq:UCov}
	\end{equation}
with $\chi$ as in \cref{prop:bexists}.
\end{cor}

\begin{proof}
	We can construct $U$ as the limit of the $u^{[M]}$ in an appropriate
	spatially- and temporally-weighted $L^{2}$ space using \cref{prop:uMCauchy}
	and Fubini's theorem. The limit preserves the expectation, so $\mathbf{E}U\equiv 1$. Positivity, spatial and temporal stationarity, and the fact that solving the SPDE \cref{eq:utdef} (i.e. solving the integral equation \cref{eq:generalizedsolution}) passes to the limit, are all clear. Finally, \cref{eq:UCov} follows directly from the convergence of $u^{[M]}(t,x)\to U(t,x)$ in  $L^2(\Omega)$ and \eqref{StMconverges}.
\end{proof}

\begin{rem}\label{rem:environment}
The spacetime stationary random field $U$  solves \eqref{utdef}, which is related to the Fokker-Planck equation for the process of ``environment seen from the particle.'' If we use $U(0,0)$ as the Radon-Nikodym derivative to tilt the probability measure $\mathbf{P}$, then the new measure is an invariant measure for the ``environment seen from the particle.'' Actually, modulo notation, \cref{eq:utdef} starting from $u(0)\equiv1$ is precisely the equation for the Radon-Nikodym derivative of the environmental process (see e.g.\ \cite[equation~(3.2)]{komorowski2001homogenization} for the expression of generator of the environmental process), from which one can easily write down the evolution of its Radon-Nikodym derivative. Thus, in a sense $U$ describes the steady state of the environmental process and the function $\chi$ in \eqref{UCov} is related to the mixing property of the steady state. For a model of random walk in balanced random environment, \cite{DeGu19} proved a similar result as \cref{thm:maintheorem}. The $U(t,x)$ constructed above corresponds to the $\rho_\omega(x,t)$ defined in \cite[Page 3]{DeGu19}, and the SPDE (\ref{eq:utdef}) corresponds to \cite[Equation (3)]{DeGu19}. For a Markovian velocity field with a large spectral gap, the invariant measure was constructed in \cite{komorowski2003invariant}, analogous to our construction of the spacetime stationary solution to \eqref{utdef}, although the velocity field here is white in time which corresponds to an infinite spectral gap.
\end{rem}

\begin{cor}
	\label{cor:theconvergenceweneed}For any sequence $((t_{k},M_{k}))_{k\ge0}$
	so that $t_{k}+M_{k}\to\infty$ as $k\to\infty$, we have for any
	$\eps>0$ (or $\eps=0$ if $d=1$) that
	\begin{equation}
		\lim_{k\to\infty}(t_{k}+M_{k})^{(d-\eps)/2} \mathbf{E}|u^{[M_{k}]}(t_k,x)-U(t_k,x)|^{2}=0.\label{eq:theconvergenceweneed}
	\end{equation}
\end{cor}

\begin{proof}
	This is clear from the translation-invariance and \cref{cor:Uexists}.
\end{proof}

Finally, we can derive a bound regarding the temporal decorrelation of $U$.
\begin{prop}\label{prop:timecorr}For any $\eps>0$ (or $\eps = 0$ if $d=1$) we have a constant $C<\infty$ such that for any $t,s\in\mathbf{R}$ and $x,y\in\mathbf{R}^d$, we have
\begin{equation}
 |\mathbf{E} [(U(t,x)-1)(U(s,y)-1)]|\le C(1+|t-s|)^{-(d-\eps)/4}.\label{eq:timedecorrelation}
\end{equation}
\end{prop}
\begin{proof}
By stationarity, we can assume without loss of generality that $t=-M$ and $s=0$, for some $M>0$. Note that the fields $u^{[M]}(0)$ and $U(-M)$ are independent. Therefore, we have
\begin{align*}
 |\mathbf{E} [(U(-M,x)-1)(U(0,y)-1)]|&\le |\mathbf{E} [(U(-M,x)-1)(U(0,y)-u^{[M]}(0,y))]|\\&\qquad+|\mathbf{E} [(U(-M,x)-1)(u^{[M]}(0,y)-1)]|\\&\le (\mathbf{E} |U(-M,x)-1|^2)^{1/2}(\mathbf{E}|U(0,y)-u^{[M]}(0,y)|^2)^{1/2}.
\end{align*}
The first expectation on the right side is bounded uniformly in $M$, and the second goes to zero when multiplied by $M^{(d-\eps)/4}$ by \cref{cor:Uexists}. This completes the proof.
\end{proof}

\section{Proof of Theorem~\ref{thm:maintheorem}}\label{sec:deltaic}
In this section we prove \cref{thm:maintheorem}. The strategy is similar to the construction of $U$ in the previous section, involving turning off the noise except on a layer around the final time. First we show that the error incurred by turning off the noise in this way is small. Then we show that, since the solution to the heat equation is smooth, the resulting solution at a point is well-approximated by a solution started at an appropriate constant (the solution to the heat equation at the same point). Due to the linearity of the SPDE, this gives us the multiplicative structure in \cref{eq:mainthm-bound}.

Let $u(t)$ solve \cref{eq:utdef} with $u(0)=\delta_{0}$, a Dirac
delta measure at zero. For any $q\leq t$, let $\mathcal{F}_{q,t}$ be the $\sigma$-algebra generated by $V(s)-V(r)$ for $q\le r\le s\le t$. Given $q>0$, define
\[
	\tilde{u}_{q}(t)=\mathbf{E}[u(t)\mid\mathcal{F}_{q,t}], \quad\quad t\geq q.
\]
Then $t\mapsto\tilde{u}_{q}(t)$ satisfies \cref{eq:utdef} for $t>q$, and  we have the initial condition
\[
	\tilde{u}_{q}(q)=G_{q},
\] where we recall that $G$ solves \cref{eq:eqG}. An important step in the proof of \cref{thm:maintheorem} is the following proposition.
\begin{prop}
	\label{prop:turnoffnoise}
	For any $\beta\in(0,1)$, there exists a constant $C=C(R,\nu,\beta)<\infty$
	so that, for all $t\ge C$,
	\begin{equation}
		\sup_{x\in\mathbf{R}^{d}}\mathbf{E}|u(t,x)-\tilde{u}_{t-t^{\beta}}(t,x)|^{2}\le  Ct^{-(d\wedge (d/2+1))-\beta d/2}\log t.%
		\label{eq:takeqtminuslogt}
	\end{equation}
\end{prop}

\begin{proof}

	\emph{Step 1: taking second moments. }Define
	\[
		\overline{A}=\begin{pmatrix}\frac{1}{4}\left[\nu I_{d}+R(0)\right] \\
			 & \nu I_{d}+R(0)
		\end{pmatrix},\qquad\tilde{A}(z)=\begin{pmatrix}\frac{1}{8}\left[R(z)+R(z)^{\mathrm{T}}]\right] & \frac{1}{4}\left[R(z)-R(z)^{\mathrm{T}}\right]  \\
               \frac{1}{4}\left[R(z)^{\mathrm{T}}-R(z)\right]  & -\frac{1}{2}\left[R(z)+R(z)^{\mathrm{T}}\right]
		\end{pmatrix},
	\]
	so we can decompose \cref{eq:Adef} as
	\begin{equation}
		A(z)=\overline{A}+\tilde{A}(z).\label{eq:AAbarAtilde}
	\end{equation}
	Let $H_{t}$ be the  solution to the PDE 
	\begin{align*}
			\partial_{t}H_{t}(\omega) & =\tr[\overline{A}\nabla^{2}H_{t}(\omega)]; \\
			H_{0}                     & =\delta_{0}.
		\end{align*}
	This means that $H_{t}(w,z)=G_{t/2}(w)G_{2t}(z)$. For any $q>0$, if we define
	\[
		S_{t}(y,z)=\mathbf{E}u(t,y+z/2)u(t,y-z/2),\qquad\tilde{S}_{q,t}(y,z)=\mathbf{E}\tilde{u}_{q}(t,y+z/2)\tilde{u}_{q}(t,y-z/2),
	\]
	then $S_{t}$ satisfies \cref{eq:S-PDE} with initial condition
	\begin{equation}
		S_0(y,z)=\delta_{0}(y)\delta_{0}(z)\label{eq:Stic}
	\end{equation}
	and $t\mapsto\tilde{S}_{q,t}$ satisfies \cref{eq:S-PDE} for $t>q$
	with initial condition
	\begin{equation}
		\tilde{S}_{q,q}=H_{q}.\label{eq:Stildeqq}
	\end{equation}
	In particular, we have
	\begin{equation}
		S_{t}(\omega)=\Gamma_{t}(\omega;0).\label{eq:Stisgamma}
	\end{equation}
	Then define
	\[
		\overline{S}_{q,t}(y,z)=\mathbf{E}(u(t,y+z/2)-\tilde{u}_{q}(t,y+z/2))(u(t,y-z/2)-\tilde{u}_{q}(t,y-z/2)).
	\]
	Again $t\mapsto\overline{S}_{q,t}$ satisfies \cref{eq:S-PDE} in $t>q$.
	The initial condition is
	\begin{align}
		\overline{S}_{q,q}(y,z) & =\mathbf{E}\left[\left(u(q,y+z/2)-\tilde{u}_{q}(q,y+z/2)\right)\left(u(q,y-z/2)-\tilde{u}_{q}(q,y-z/2)\right)\right]\nonumber \\
		                        & =\mathbf{E}\left[\left(u(q,y+z/2)-G_{q}(y+z/2)\right)\left(u(q,y-z/2)-G_{q}(y-z/2)\right)\right]\nonumber                     \\
		                        & =\mathbf{E}[u(q,y+z/2)u(q,y-z/2)]-G_{q}(y+z/2)G_{q}(y-z/2)\nonumber                                                \\
		                        & =S_{q}(y,z)-\tilde{S}_{q,q}(y,z),\label{eq:Sbarisdiff}
	\end{align}
	where we used the fact that $\mathbf{E}u(t)=G_t$. Therefore, by the linearity of \cref{eq:S-PDE} we in fact have
	\[
		\overline{S}_{q,t}(y,z)=S_{t}(y,z)-\tilde{S}_{q,t}(y,z)
	\]
	for all $t>q$.

	\emph{Step 2: proving the bound. }Similar to the proof of \cref{prop:uMCauchy},
	our goal is to prove an upper bound on $\overline{S}_{q,t}(y,0)$.
	Since $t\mapsto\overline{S}_{q,t}$ satisfies \cref{eq:S-PDE} we have
	the identity
	\begin{equation}
		\overline{S}_{q,t}(\omega)=\int\Gamma_{t-q}(\omega;\omega')\overline{S}_{q,q}(\omega')\,\dif\omega'.\label{eq:Sbarformula}
	\end{equation}
	By the Duhamel principle applied to the PDE \cref{eq:S-PDE}, using
	the decomposition \cref{eq:AAbarAtilde}, we have
	\[
		S_{q}(\omega')=H_{q}(\omega')+ \int_{0}^{q}\int H_{q-s}(\omega'-\omega'')\tr\left[\nabla^{2}[\tilde{A}S_{s}](\omega'')\right]\,\dif\omega''\,\dif s.
	\]
	Subtracting \cref{eq:Stildeqq} and recalling \cref{eq:Sbarisdiff}, we
	obtain
	\[
		\overline{S}_{q,q}(\omega')= \int_{0}^{q}\int H_{q-s}(\omega'-\omega'')\tr\left[\nabla^{2}[\tilde{A}S_{s}](\omega'')\right]\,\dif\omega''\,\dif s.
	\]
	Now plugging this into \cref{eq:Sbarformula} and using Fubini's theorem,
	we obtain
	\begin{equation}
		\overline{S}_{q,t}(\omega)= \int_{0}^{q}\int K_{q-s,t-q}(\omega;\omega')\tr\left[\nabla^{2}[\tilde{A}S_{s}](\omega')\right]\,\dif\omega'\,\dif s,\label{eq:sbaralltogether}
	\end{equation}
	where we have defined
	\begin{equation}
		K_{r_{1},r_{2}}(\omega;\omega')=\int\Gamma_{r_{2}}(\omega;\omega'')H_{r_{1}}(\omega''-\omega')\,\dif\omega''.\label{eq:Krsdef}
	\end{equation}
	Integrating by parts in \cref{eq:sbaralltogether}, we have
	\begin{equation}
		\overline{S}_{q,t}(\omega)=\int_{0}^{q}\int S_{s}(\omega')\tr\left[\tilde{A}(\omega')\nabla_{\omega'}^{2}K_{q-s,t-q}(\omega;\omega')\right]\,\dif\omega'\,\dif s.
		\label{eq:sbarIBP}
	\end{equation}
	Using \cref{lem:Kderivs-1} below, and also another application of \cref{lem:fundsolnbound}
	(and \cref{eq:Stisgamma}) to bound $S_{s}(\omega')$, in \cref{eq:sbarIBP},
	we obtain, for $t\ge q+1$,
	\begin{equation}
		\begin{aligned}
			\left|\overline{S}_{q,t}(\omega)\right| & \le C\int_{0}^{q}[(q-s)^{-1}\wedge 1](t-s)^{-d}s^{-d}\int|\tilde{A}(z')|\exp\left\{ -\frac{|\omega'-\omega|^{2}}{C(t-s)}-\frac{|\omega'|^{2}}{Cs}\right\} \,\dif\omega'\,\dif s\label{eq:Sbarintegral} \\
			                                        & =C\left(\int_{0}^1+\int_1^q\right)[(q-s)^{-1}\wedge 1](t-s)^{-d}s^{-d}\int|\tilde{A}(z')|\exp\left\{ -\frac{|\omega'-\omega|^{2}}{C(t-s)}-\frac{|\omega'|^{2}}{Cs}\right\} \,\dif\omega'\,\dif s       \\
			                                        & \coloneqq I_1+I_2.
		\end{aligned}
	\end{equation}
	To control the above integral, we consider the region of $s\in(0,1)$ and $s\in(1,q)$ separately. For the integration in  $s\in(0,1)$, by the fact that $\tilde{A}$ is uniformly bounded, we integrate in $\omega'$ to derive
	\begin{equation}\label{eq:boundI1}
		\begin{aligned}
			I_1\leq Ct^{-d}\int_{0}^{1}[(q-s)^{-1}\wedge 1] \dif s \leq Ct^{-d}(q-1)^{-1}.
		\end{aligned}
	\end{equation}
	For the integration in $s\in(1,q)$, to control the inner integral, we write
	\begin{align*}
		\int & |\tilde{A}(z')|\exp\left\{ -\frac{|\omega'-\omega|^{2}}{C(t-s)}-\frac{|\omega'|^{2}}{Cs}\right\} \,\dif\omega'                                                                                             \\
		     & =\left(\int|\tilde{A}(z')|\exp\left\{ -\frac{|z'-z|^{2}}{C(t-s)}-\frac{|z'|^{2}}{Cs}\right\} \,\dif z'\right)\left(\int\exp\left\{ -\frac{|y'-y|^{2}}{C(t-s)}-\frac{|y'|^{2}}{Cs}\right\} \,\dif y'\right) \\
		     & \le C\|\tilde{A}\|_{L^{1}(\mathbf{R}^{d};\mathbf{R}^{2d\times2d})}\left(\frac{s(t-s)}{t}\right)^{d/2}                                                                                                      %
	\end{align*}
	for a new constant $C$, still depending only on $R$ and $\nu$.
	Using this bound in \cref{eq:Sbarintegral}, we obtain
	\begin{align}
		I_2 & \le C t^{-d/2}\int_{1}^{q}[(q-s)^{-1}\wedge 1](t-s)^{-d/2}s^{-d/2}\,\dif s\nonumber                    \\
		    & \le C t^{-d/2}(t-q)^{-d/2}\int_{1}^{q}[(q-s)^{-1}\wedge 1]s^{-d/2}\,\dif s.\label{eq:justintegralleft}
	\end{align}
	Now we estimate the last integral in two parts. First we have
	\begin{align*}
		\int_{q/2}^{q}[(q-s)^{-1}\wedge 1]s^{-d/2}\,\dif s & \le(q/2)^{-d/2}\int_{q/2}^{q}[(q-s)^{-1}\wedge 1]\,\dif s=(q/2)^{-d/2}[1+\log(q/2)].
	\end{align*}
	Second, we have
	\[
		\int_{1}^{q/2}(q-s)^{-1}s^{-d/2}\,\dif s\le2q^{-1}\int_{1}^{q/2}s^{-d/2}\,\dif s\le Cq^{-1}\bigg(\sqrt{q}\1_{d=1}+\log q\1_{d=2}+\1_{d\geq3}\bigg)\le Cq^{-((d/2)\wedge 1)}\log q.
	\]
	Using the last two inequalities in \cref{eq:justintegralleft} and taking $q=t-t^\beta$  we
	obtain, for a $C$ now depending also on $\beta$,
	\[
		I_2 \le Ct^{-(d\wedge (d/2+1))-\beta d/2}\log t.
	\]
	Combining this with \cref{eq:boundI1}, we obtain \cref{eq:takeqtminuslogt}.
\end{proof}
Now we must prove the lemma we used in the previous proof.
\begin{lem}
	\label{lem:Kderivs-1}Recall the definition \cref{eq:Krsdef} of $K_{r_{1},r_{2}}$.
	There is a constant $C=C(\nu,R)<\infty$ so that, for all $r_{1}>0$ and $r_{2}\ge 1$, we have
	\begin{equation}
		\left|\nabla_{\omega'}^{2}K_{r_{1},r_{2}}(\omega;\omega')\right|\le C(r_{1}^{-1}\wedge 1)(r_{1}+r_{2})^{-d}\exp\left\{ -\frac{|\omega'-\omega|^{2}}{C(r_{1}+r_{2})}\right\} .\label{eq:Kderivestimate}
	\end{equation}

\end{lem}

\begin{proof}
	Differentiating \cref{eq:Krsdef} and using \cref{lem:fundsolnbound},
	we have
	\begin{align}
		\left|\nabla_{\omega'}^{2}K_{r_{1},r_{2}}(\omega;\omega')\right| & \le\int\Gamma_{r_{2}}(\omega;\omega'')\left|\nabla_{\omega'}^{2}H_{r_{1}}(\omega''-\omega')\right|\,\dif\omega''\nonumber                                \\
		                                                                 & \le Cr_{2}^{-d}r_{1}^{-d-1}\int\exp\left\{ -C^{-1}r_{2}^{-1}|\omega-\omega''|^{2}-C^{-1}r_{1}^{-1}|\omega''-\omega'|^{2}\right\} \,\dif\omega''\nonumber \\
		                                                                 & \le Cr_{1}^{-1}(r_{1}+r_{2})^{-d}\exp\left\{ -C^{-1}(r_{1}+r_{2})^{-1}|\omega-\omega'|^{2}\right\},\label{eq:differentiateH}
	\end{align}
	where we allowed the constant $C$ to change from line to line. Alternatively,
	we can use integration by parts and \cref{lem:fundsolnderivbound} to derive
	that
	\begin{align}
		\left|\nabla_{\omega'}^{2}K_{r_{1},r_{2}}(\omega;\omega')\right| & \le\int\left|\nabla_{\omega''}^{2}\Gamma_{r_{2}}(\omega;\omega'')\right|H_{r_{1}}(\omega''-\omega')\,\dif\omega''\nonumber                             \\
		                                                                 & \le Cr_{2}^{-d}r_{1}^{-d}\int\exp\left\{ -C^{-1}r_{2}^{-1}|\omega-\omega''|^{2}-C^{-1}r_{1}^{-1}|\omega''-\omega'|^{2}\right\} \,\dif\omega''\nonumber \\
		                                                                 & \le C(r_{1}+r_{2})^{-d}\exp\left\{ -C^{-1}(r_{1}+r_{2})^{-1}|\omega-\omega'|^{2}\right\},\label{eq:differentiategamma}
	\end{align}
	where again $C$ changed from line to line.
	Together, \cref{eq:differentiateH} and \cref{eq:differentiategamma}
	imply \cref{eq:Kderivestimate}.
\end{proof}

Now we want to show that, when $1\ll t-q\ll t$, the field $\tilde{u}_{q}(t)$ is well-approximated by
the stationary solution $U(t)$ multiplied by $G_t$. Let $\underline{u}_{q}$ solve \cref{eq:utdef}
for $t>q$, with initial condition
\[
	\underline{u}_{q}(q)\equiv1,
\] so by \cref{cor:theconvergenceweneed}
we have%

\begin{equation}
	\adjustlimits\lim_{t-q\to\infty}\sup_{x\in\mathbf{R}^{d}}\mathbf{E}|\underline{u}_{q}(t,x)-U(t,x)|^{2}=0.\label{eq:underlineUapproachesstationary}
\end{equation}

\begin{prop}
	There exists a constant $C$ so that, for any $x\in\mathbf{R}^{d}$ and $t>q$, we
	have
	\begin{equation}
		\mathbf{E}|\tilde{u}_{q}(t,x)-G_{q}(x)\underline{u}_{q}(t,x)|^{2}\le Cq^{-d-1}(t-q),\label{eq:approxbyconst}
	\end{equation}
	and in particular, for any $\beta\in(0,1)$, there exists a constant $C=C(R,\nu,\beta)$
	such that for all $t>1$
	\begin{equation}
		\sup_{x\in\mathbf{R}^{d}}\mathbf{E}\left|\tilde{u}_{t-t^{\beta}}(t,x)-G_{t-t^{\beta}}(x)\underline{u}_{t-t^{\beta}}(t,x)\right|^{2}\le Ct^{-d-1+\beta}.\label{eq:approxbyconst-1}
	\end{equation}
\end{prop}

\begin{proof}
	Fix $q,x$. First recall that $\tilde{u}_q, \underline{u}_q$ both solve \cref{eq:utdef} in $t>q$. Let
	\[
		\underline{S}_{q,t,x}(y,z)=\mathbf{E}(\tilde{u}_{q}(t,y+z/2)-G_{q}(x)\underline{u}_{q}(t,y+z/2))(\tilde{u}_{q}(t,y-z/2)-G_{q}(x)\underline{u}_{q}(t,y-z/2)).
	\]
	Then as a function of $(t,y,z)$, we have $\underline{S}_{q,t,x}$ solves \cref{eq:S-PDE} with initial condition
	\[
		\underline{S}_{q,q,x}(y,z)=\left(G_{q}(y+z/2)-G_{q}(x)\right)\left(G_{q}(y-z/2)-G_{q}(x)\right).
	\]
	Therefore, we have
	\[
		\underline{S}_{q,t,x}(y,z)=\int\left(G_{q}(y'+z'/2)-G_{q}(x)\right)\left(G_{q}(y'-z'/2)-G_{q}(x)\right)\Gamma_{t-q}(y,z;y',z')\,\dif y'\,\dif z',
	\]
	and so
	\begin{align*}
		\underline{S}_{q,t,x}(x,0) & =\int\left|G_{q}(y'+z'/2)-G_{q}(x)\right|\left|G_{q}(y'-z'/2)-G_{q}(x)\right|\Gamma_{t-q}(x,0;y',z')\,\dif y'\,\dif z'                                \\
		                           & \le C(t-q)^{-d}\int\left|G_{q}(y'+z'/2)-G_{q}(x)\right|\left|G_{q}(y'-z'/2)-G_{q}(x)\right|\e^{-\frac{|y'-x|^{2}+|z'|^{2}}{C(t-q)}}\,\dif y'\,\dif z' \\
		                           & \le C(t-q)^{-d}\|\nabla G_{q}\|_{\infty}^{2}\int|y'+z'/2|\cdot|y'-z'/2|\e^{-\frac{|y'|^{2}+|z'|^{2}}{C(t-q)}}\,\dif y'\,\dif z'                       \\
		                           & \le Cq^{-d-1}(t-q)\int|y'+z'/2|\cdot|y'-z'/2|\e^{-\frac{|y'|^{2}+|z'|^{2}}{C}}\,\dif y'\,\dif z',
	\end{align*}
	where in the first inequality we used \cref{lem:fundsolnbound}. This
	completes the proof of \cref{eq:approxbyconst}, and \cref{eq:approxbyconst-1}
	follows immediately.
\end{proof}
Now we can prove our main theorem.
\begin{proof}[Proof of \cref{thm:maintheorem}]
	Fix $\beta\in(0,1)$ and let $q=t-t^{\beta}$. We use the triangle
	inequality to write
	\begin{equation}
		\begin{aligned}\mathbf{E}|u(t,x)-G_{t}(x)U(t,x)|^{2} & \le C\mathbf{E}|u(t,x)-\tilde{u}_{q}(t,x)|^{2}+C\mathbf{E}|\tilde{u}_{q}(t,x)-G_{q}(x)\underline{u}_{q}(t,x)|^{2}              \\
                                                     & \quad+C|G_{q}(x)-G_{t}(x)|^{2}\mathbf{E}\underline{u}_{q}(t,x)^{2}+CG_{t}(x)^{2}\mathbf{E}|\underline{u}_{q}(t,x)-U(t,x)|^{2}.
		\end{aligned}
		\label{eq:splitupfinalthing}
	\end{equation}
	We note that
	\begin{align}
		\sup_{x\in\mathbf{R}^{d}}|G_{q}(x)-G_{t}(x)| & \le(t-q)\sup_{x\in\mathbf{R}^{d},s\in[q,t]}|\partial_{s}G_{s}(x)|\le Ct^{\beta-d/2-1}.\label{eq:fudgethegaussian}
	\end{align}
	Applying \cref{eq:takeqtminuslogt}, \cref{eq:approxbyconst-1}, \cref{eq:fudgethegaussian}
	(along with the fact that $\mathbf{E}\underline{u}_{q}(t,x)^{2}$
	is uniformly bounded by \cref{cor:theconvergenceweneed}), and \cref{eq:underlineUapproachesstationary},
	respectively, to the four terms on the right side of \cref{eq:splitupfinalthing},
	we obtain for every $\eps>0$ (or $\eps=0$ if $d=1$), there is a constant
	$C=C(R,\nu,\beta,\eps)<\infty$ so that
	\[
		\begin{aligned}\mathbf{E}|u(t,x)-G_{t}(x)U(t,x)|^{2} & \le C\left(t^{-(d\wedge (d/2+1))-\beta d/2}\log t+t^{-d-1+\beta}+t^{2\beta-d-2}+t^{-d-\beta d/2+\beta\eps}\right),\end{aligned}
	\]
	Then we take
	\[
		\beta=\tfrac23 \1_{d=1}+\tfrac{d}{d+2}\1_{d\ge 2}
	\]
	to further derive that
	\[
        \mathbf{E}|u(t,x)-G_{t}(x)U(t,x)|^{2}\leq  Ct^{-d}\bigg(t^{-1/3}\1_{d=1}+t^{-2/(2+d)+\eps d/(2+d)}\1_{d\ge 2}\bigg)\log t.
	\]
	Changing $\eps$ yields \cref{eq:mainthm-bound}, and \cref{eq:takethesup} is then a consequence of the formula for the Gaussian density.
	
	If $\sum_{i=1}^d \frac{\partial R_{ij}}{\partial x_i}\equiv 0$ for each $j$, then by \cref{rem:incompressiblechi}, \cref{eq:UCov} and the fact that $\mathbf{E} U\equiv 1$,  we have $U\equiv 1$ almost surely. On the other hand, if $V$ is not incompressible, then it is clear that the constant $1$ does not solve \cref{eq:utdef}, and so $U$ cannot be a.s.\ identically equal to $1$ by \cref{cor:Uexists}. Finally, \cref{eq:Ucorrbd} follows from \cref{eq:UCov} and \cref{eq:powerlaw}, and \cref{eq:Utimecorrbd} follows from \cref{prop:timecorr}. This completes the proof of the theorem.
	\end{proof}
\bibliographystyle{hplain-ajd}
\bibliography{rwre}

\end{document}